\documentclass[a4paper,12pt]{article}

\usepackage[left=2cm,right=2cm, top=2cm,bottom=3cm,bindingoffset=0cm]{geometry}

\usepackage{verbatim}
\usepackage{amsmath}
\usepackage{amsthm}
\usepackage{amssymb}
\usepackage{delarray}
\usepackage{cite}
\usepackage{hyperref}
\usepackage{mathrsfs}
\usepackage{tikz}
\usetikzlibrary{patterns}
\usepackage{caption}
\DeclareCaptionLabelSeparator{dot}{. }
\newcommand{\al}{\alpha}
\newcommand{\be}{\beta}

\newcommand{\de}{\delta}
\newcommand{\la}{\lambda}
\newcommand{\om}{\omega}

\newcommand{\eps}{\varepsilon}
\newcommand{\vv}{\varphi}

\theoremstyle{plain}

\numberwithin{equation}{section}

\newtheorem{thm}{Theorem}[section]
\newtheorem{lem}[thm]{Lemma}
\newtheorem{prop}[thm]{Proposition}
\newtheorem{cor}[thm]{Corollary}

\theoremstyle{definition}

\newtheorem{example}[thm]{Example}
\newtheorem{ip}[thm]{Inverse Problem}
\newtheorem{defin}[thm]{Definition}

\theoremstyle{remark}

\newtheorem{remark}[thm]{Remark}

\DeclareMathOperator*{\Res}{Res}
\DeclareMathOperator{\rank}{rank}

 \allowdisplaybreaks

\begin{document}

\begin{center}
{\Large\bf Uniform stability for the matrix inverse \\[0.2cm] Sturm-Liouville problems}
\\[0.5cm]
{\bf Natalia P. Bondarenko}
\end{center}

\vspace{0.5cm}

{\bf Abstract.} In this paper, the uniform stability of the inverse spectral problem is proved for the matrix Sturm-Liouville operator on a finite interval. Namely, we describe the sets of spectral data, on which the inverse spectral mapping is bounded and, consequently, the uniform estimates hold for the differences of the matrix potentials and of the corresponding coefficients of the boundary conditions. Our approach is based on a constructive procedure for solving the inverse problem by developing ideas of the method of spectral mappings. In addition, we apply our technique to obtain the uniform stability of the inverse Sturm-Liouville problem on the star-shaped graph.

\medskip


\medskip


\section{Introduction} \label{sec:intr}

Consider the matrix Sturm-Liouville equation
\begin{equation} \label{eqv}
-Y'' + Q(x) Y = \la Y, \quad x \in (0,\pi),
\end{equation}
with the boundary conditions
\begin{equation} \label{bc}
Y'(0) - h Y(0) = 0, \quad Y'(\pi) + H Y(\pi) = 0,
\end{equation}
where $Y = Y(x)$ is a column vector function of size $m$ ($m \ge 1$), $Q(x)$ is a complex-valued $(m \times m)$ matrix function with elements of $L_2(0,\pi)$, which is called the \textit{potential}, $\la$ is the spectral parameter, $h$ and $H$ are constant matrices of $\mathbb C^{m \times m}$. Assume that
\begin{equation} \label{sa}
Q^*(x) = Q(x) \:\: \text{a.e. on} \:\: (0,\pi), \quad h = h^*, \quad H = H^*,
\end{equation}
where the symbol ``$^*$'' denotes the conjugate transpose. Then, the problem \eqref{eqv}--\eqref{bc} is self-adjoint. 

In this paper, we focus on the inverse spectral problem, which consists in the recovery of $Q(x)$, $h$, and $H$ from the eigenvalues $\{ \la_u \}$ and the norming vectors $\{ v_u \}$, $v_u = Y_u(0)$, where $\{ Y_u(x) \}$ are orthonormal eigenfunctions of the problem \eqref{eqv}--\eqref{bc} (see Section~\ref{sec:prelim} for more details). The uniqueness of solution and the spectral data characterization for this problem follow from the previously known results \cite{Yur06, Bond11, Bond19}. The main goal of this paper is to prove the uniform stability for this inverse problem, which is of fundamental novelty for the matrix Sturm-Liouville operators.

The most complete results in the inverse problem theory were obtained for the scalar ($m = 1$) Sturm-Liouville operators (see, e.g., the monographs \cite{Mar77, Lev84, PT87, FY01, Krav20} and references therein). Matrix inverse Sturm-Liouville problems ($m > 1$) have a long history. It began with the study of inverse scattering problems by Agranovich and Marchenko~\cite{AM63} on the half-line and by Wadati and Kamijo~\cite{WK74} on the line. In the finite interval case, the uniqueness of recovering the matrix Sturm-Liouville potential from various types of spectral characteristics (Weyl matrix, several spectra, monodromy matrix, eigenvalues and weight matrices) has been investigated by Carlson~\cite{Car02}, Chabanov~\cite{Chab04}, Malamud~\cite{Mal05}, Yurko~\cite{Yur06}, and other scholars. In contrast to the scalar case, the matrix problem \eqref{eqv}--\eqref{bc} may have an infinite number of multiple or asymptotically multiple eigenvalues, which significantly complicates the analysis of inverse spectral problems.

The necessary and sufficient conditions for solvability of the matrix inverse Sturm-Liouville problems on a finite interval have been obtained in \cite{CK09, MT10, Bond11} in parallel by different methods. Specifically, Chelkak and Korotyaev \cite{CK09} have given the spectral data characterization in the case of asymptotically simple spectrum by developing the approach of Trubowitz and his co-authors (see~\cite{PT87}). That method is based on the ideas of Borg~\cite{Borg46} for local solution of an inverse problem and on transforms that change a finite amount of spectral data to achieve a global solution. For the matrix case, such transforms were constructed by Chelkak and Korotyaev in \cite{CK06}. 

Mykytyuk and Trush~\cite{MT10} reduced the matrix Sturm-Liouville equation with Miura potential to a Dirac-type system and then applied the Krein accelerant method \cite{Krein56}. Later on, analogous reduction was applied by Eckhardt et al \cite{Eckh14, Eckh15} to the matrix Sturm-Liouville operators on the half-line and on the line. It is also worth mentioning that Mykytyuk and Puyda \cite{MP12, Puy13} deduced the spectral data characterization for the matrix Dirac-type operators. 

The third approach of Bondarenko \cite{Bond11} develops the method of spectral mappings \cite{FY01, Yur02}, which was originally created for higher-order differential operators (see \cite{Leib66, Yur92}).
This method can be used for reducing various types of inverse spectral problems to linear equations in suitable Banach spaces. The central place in the method of spectral mappings is taken by contour integration of some meromorphic functions (``spectral mappings'') in the complex plane of the eigenparameter. The first steps in applying this approach to the matrix Sturm-Liouville problem \eqref{eqv}--\eqref{bc} was implemented by Yurko \cite{Yur06}. However, the results of \cite{Yur06} are limited to uniqueness theorems and a reconstruction algorithm for the case of simple eigenvalues. Bondarenko has solved the matrix inverse problem in the general case of arbitrary behavior of the spectrum, including multiple and asymptotically multiple eigenvalues. The results of \cite{Bond11}, which were subsequently improved in \cite{Bond19}, contain spectral data characterization for the self-adjoint and non-self-ajoint problems of form \eqref{eqv}--\eqref{bc}. Later on, this approach was developed for the Sturm-Liouville problems with general self-adjoint boundary conditions, including the case of singular potentials in the class of matrix functions $W_2^{-1}(0,\pi)$ (see \cite{Bond20-ipse, Bond20-amp, Bond21-mn, Bond21-amp}).

Inverse problems for the matrix Sturm-Liouville operators have a variety of applications. Such problems arise in quantum mechanics \cite{AM63} and elasticity theory \cite{BHN95}. Furthermore,  solutions of matrix inverse Sturm-Liouville problems have been used for description of electromagnetic waves \cite{BS95}, for
reconstructing the coefficients of the linear reaction diffusion system~\cite{Boum17}, and for integration of matrix nonlinear evolution equations by the inverse scattering transform (see the classical work \cite{CD77} and some recent studies \cite{DM23, KUB25} on this topic). The application, which is mostly related to this paper, is concerned with quantum graphs.

Differential operators on graphs, also called quantum graphs, are given by differential expressions on the edges of a metric graph and by matching conditions in the vertices. Such models are applied in science and engineering for describing physical processes in spacial networks. Basic theory of quantum graphs and their applications are presented, e.g., in the monographs \cite{PPP04, BCFK06, BK13, Kur24}. Various approaches to inverse problems on metric graphs were developed in \cite{Bel04, Now07, ALM10, Yur16} and other studies. In particular, one of the existing approaches consists in the representation of differential equations on graphs in the matrix form. This idea was applied by Harmer \cite{Har02, Har05} to the inverse scattering problem on the star-shaped graph with infinite edges, which is reduced to the matrix Sturm-Liouville problem on the half-line with the general self-adjoint boundary condition. The latter inverse problem has been extensively studied by Aktosun and Weder \cite{AW21}. Moreover, Xu and Bondarenko \cite{XB22} have recently proved its stability. The mentioned studies \cite{Har02, Har05, AW21, XB22} rely on the methods of Marchenko \cite{Mar77, AM63} and deal with a finite discrete spectrum. Thus, in this context, the case of the half-line is simpler than the finite interval case, which is concerned with a countable set of eigenvalues. Matrix inverse Sturm-Liouville problems with general self-adjoint boundary conditions started to be investigated in \cite{CM13, Xu19}, but those results were limited to uniqueness theorems. Later, Bondarenko \cite{Bond20-ipse} has developed a constructive procedure to find a solution based on the method of spectral mappings. This approach allowed her to get the spectral data characterization for the Sturm-Liouville operator on the star-shaped graph and for a more general matrix Sturm-Liouville operator with non-diagonal potential (see \cite{Bond20-amp}). Furthermore, those results were generalized to distributional potentials and to graphs of arbitrary geometric structure with rationally dependent edge lengths (see \cite{Bond21-mn, Bond21-amp}).


In recent years, significant progress has been achieved in the investigation of the uniform stability for inverse spectral problems. In particular, Savchuk and Shkalikov \cite{SS10, SS13} have proved the unconditional uniform stability of the inverse Sturm-Liouville problems with potentials in the scale of the Sobolev spaces $W_2^{\theta}$, $\theta > -1$. The case $\theta = -1$ was studied by Hryniv \cite{Hryn11-1} by another method.
Their fundamental achievement consists in describing the sets in the space of spectral data, on which the inverse spectral mapping is bounded and, consequently, the stability estimates are uniform. This issue is quite non-trivial, since the inverse problem is nonlinear and the spaces of spectral data and of the potentials are infinite dimensional. The results of \cite{SS10, SS13} allowed Savchuk and Shkalikov to deduce the stability estimates for approximations of the potential by finite spectral data (see \cite{SS14, Sav16}). Such estimates are especially important for practical applications since only a finite amount of spectral data is usually available in practice. Furthermore, the uniform stability of inverse problems was studied for the Dirac system \cite{Hryn11-2}, for several classes of nonlocal operators \cite{But21, BD22, Kuz23}, for non-self-adjoint Sturm-Liouville operators \cite{GMXA23, Bond24}, for the half-inverse Sturm-Liouville problem \cite{Bond25-mn}, and for the Sturm-Liouville problem with rational Herglotz-Nevanlinna functions of the eigenparameter in the boundary conditions \cite{Bond25-jmp}. However, for the matrix Sturm-Liouville operators, the question about constraints on spectral data that guarantee the uniform stability of the inverse problem was open. 


Unconditional uniform stability of inverse problems on graphs, to the best of the author's knowledge, was also not investigated before.
Some local stability results were obtained in \cite{Bond20-ipse, Bond18, Bond21-mmas}. In addition, the conditional uniform stability (i.e. with restrictions on the potentials) has been proved for inverse problems on graphs with a loop \cite{MT19, Bond25-jiip} and on arbitrary compact trees \cite{Bond25-sam}. 
We also mention that Buterin \cite{But23} obtained the uniform stability of the inverse problem on a graph for a functional-differential operator, which is nonlocal and so fundamentally different from differential operators.
At the same time, numerical methods for solving inverse spectral problems on quantum graphs are actively developing (see \cite{AK23, AKK24}). This indicates the importance of studying the stability of such problems.

In this paper, the matrix Sturm-Liouville operator given by \eqref{eqv}--\eqref{bc} is studied. We aim to find the constraints on the spectral data that guarantee the uniform stability of the inverse spectral problem. For this purpose, we first investigate the direct problem that consists in determining the eigenvalues $\{ \la_u \}$ and the norming vectors $\{ v_u \}$ by the problem parameters $Q$, $h$, and $H$. We describe the spectral image of the ball $\| Q \|_{L_2} + \| h \| + \| H \| \le R$ in Theorem~\ref{thm:bounddir}.     An important role in our study is played by the vector functional sequence $\{ v_u \cos(\sqrt \la_u x) \}$, whose completeness is crucial in the spectral data characterization (see \cite{Bond19}). We show that, under the uniform boundedness of $Q$, $h$, and $H$, the sequences $\{ v_u \cos(\sqrt \la_u x) \}$ form a uniformly bounded family of Riesz bases. The proof is based on the limiting approach and on using the compact embedding $W_2^1[0,\pi] \subset W_2^{\al}[0,\pi]$, $\al < 1$, by developing the ideas of \cite{SS10}.

Next, we prove Theorem~\ref{thm:boundinv} on the uniform boundedness of the inverse problem for spectral data on the set $\mathcal S_{\Omega, \eps}$ described by the assertions of Theorem~\ref{thm:bounddir}. We rely on the ideas of the method of spectral mappings \cite{FY01, Yur02}, which has been developed for the matrix case in \cite{Yur06, Bond11, Bond19} and subsequent studies. The inverse problem is reduced to a linear main equation in the Banach space of infinite matrix sequences.
However, it is  inconvenient to use the construction of the main equation from \cite{Bond11, Bond19} for investigation of the uniform stability. Therein, the dependence of the operator and the right-hand side of the main equation on the spectral data is not continuous. Therefore, we apply the modification of the main equation that has been introduced in \cite{Bond24} for the scalar case to achieve the continuity. Within the framework of our method, the main equation is essential for proving the uniform boundedness of the inverse problem. Namely, we show that, if spectral data lie in $S_{\Omega,\eps}$, then the inverse operator from the main equation is uniformly bounded. For the proof, we pass to a weak limit $\{ \la_u, v_u \}$ of a sequence $\{ \la_u^{(p)}, v_u^{(p)} \}$ in $S_{\Omega,\eps}$ as $p \to \infty$. The limit falls out of the spectral data class corresponding to matrix potentials $Q \in L_2$. Nevertheless, we show that it corresponds to some matrix potential of $W_2^{-1}$ and rely on the inverse problem theory for such operators constructed in \cite{Bond21-mn, Bond21-amp}.

The main theorem on the uniform stability (Theorem~\ref{thm:unistab}) is proved by applying the uniform boundedness (Theorem~\ref{thm:boundinv}) and using the reconstruction formulas for the potential $Q$ and the coefficients $h$ and $H$ of the boundary conditions. A feature of our approach is partitioning the eigenvalues of two problems into groups and estimating the differences of two potentials in terms of some spectral data combinations in accordance with this partition. This technique expands the applications of our results, since in the matrix case multiple eigenvalues can split under small perturbations of the spectrum. Moreover, even for two problems having only simple eigenvalues, a non-trivial partition of their eigenvalues can be needed for their comparison. We provide a series of examples to illustrate our approach. In particular, we prove the convergence of finite spectral data approximations to $Q$, $h$, and $H$.

Furthermore, our technique is applied to the inverse Sturm-Liouville problem on the star-shaped graph. In order to maintain continuity with the previous studies \cite{Bond20-ipse, Bond20-amp}, we consider the operator with the Dirichlet boundary conditions and with the standard $\de$-type matching conditions at the internal vertex. The corresponding boundary value problem is represented in the matrix form with the Dirichlet boundary condition at $x = 0$ and the self-adjoint boundary condition of general type at $x = \pi$. Although this problem differs from \eqref{eqv}--\eqref{bc}, our method for studying the uniform stability is transferred to the graph case with a number of technical modifications. As a result, we get the uniform boundedness for the direct and inverse spectral problems, as well as the unconditional uniform stability of the inverse problem for the Sturm-Liouville operator on the star-shaped graph. To the best of the author's knowledge, the results of this paper are first of this kind for differential operators on metric graphs.

The paper is organized as follows. Section~\ref{sec:prelim} contains the inverse problem statement, theorems on uniqueness and spectral data characterization, and other preliminaries. Section~\ref{sec:main} presents the main results for the problem \eqref{eqv}--\eqref{bc}. They include uniform bounds for solutions of the direct and inverse spectral problems (Theorems~\ref{thm:bounddir} and~\ref{thm:boundinv}, respectively) and the uniform stability for the inverse problem (Theorem~\ref{thm:unistab}). In Section~\ref{sec:uchar}, we deduce the uniqueness of recovering $Q$, $h$, and $H$ from $\{ \la_u, v_u \}$ and characterization for this kind of spectral data from the results of the previous studies \cite{Yur06, Bond11, Bond19}. In Section~\ref{sec:bounddir}, the proof of Theorem~\ref{thm:bounddir} is given. In Section~\ref{sec:maineq}, we derive the main equation of the inverse problem in a suitable Banach space of infinite matrix sequences. Section~\ref{sec:boundinv} contains the proof of Theorem~\ref{thm:boundinv} on the uniform boundedness of the inverse problem. In Section~\ref{sec:unistab}, we prove the uniform stability theorem and illustrate it with several examples. In Section~\ref{sec:gen}, some technical restriction is removed, and the results for the general case are obtained. Finally, in Section~\ref{sec:graph}, we apply our technique to the inverse Sturm-Liouville problem on the star-shaped graph.

\section{Preliminaries} \label{sec:prelim}

Throughout the paper, we use the following \textbf{notations}:
\begin{itemize}
\item $\mathbb C^m$ is the unitary space of complex-valued column vectors of size $m$.
\item $\mathbb C^{m \times m}$ is the space of complex-valued matrices of size $(m \times m)$ with the operator norm $\| A \| = s_{max}(A)$, where $s_{max}(A)$ is the maximal singular value of a matrix $A$.
\item $I_m$ is the $(m \times m)$ unit matrix, $0_m$ is the $(m \times m)$ zero matrix.
\item $L_2((0,\pi); \mathbb C^m)$ is the Hilbert space of complex-valued $m$-vector functions with elements of $L_2(0,\pi)$ equipped with the scalar product 
$$
(Y, Z) = \sum_{j = 1}^m \int_0^{\pi} \overline{y_j(x)} z_j(x) \, dx, \quad Y = [y_j(x)]_{j = 1}^m, \quad Z = [z_j(x)]_{j = 1}^m, \quad Y, Z \in L_2((0,\pi); \mathbb C^m).
$$
and the norm $\| Y \|_{L_2} = \sqrt{(Y, Y)}$.
\item $L_2((0,\pi); \mathbb C^{m \times m})$ is the Banach space of complex-valued $(m \times m)$ matrix functions with elements of $L_2(0, \pi)$ equipped with the norm
$$
\| A \|_{L_2} := \sum_{j,k = 1}^m \| a_{jk} \|_{L_2}, \quad A = [a_{jk}]_{j,k = 1}^m \in L_2((0,\pi); \mathbb C^{m \times m}).
$$
\item More generally, if $\mathcal I$ is an interval on the real line and $\mathcal M(\mathcal I)$ is a Banach space of functions on $\mathcal I$, then the notation $\mathcal M(\mathcal I; \mathbb C^{m \times m})$ is used for the Banach space of the complex-valued matrix functions with elements of $\mathcal M(\mathcal I)$ equipped with the norm
$$
\| A \|_{\mathcal M(\mathcal I; \mathbb C^{m \times m})} := \sum_{j,k = 1}^m \| a_{jk} \|_{\mathcal M(\mathcal I)}, \quad A = [a_{jk}]_{j,k = 1}^m.
$$
\item $l_2(\mathbb C^{m \times m})$ is the space of matrix sequences
$K = \{ K_n \}$, $K_n \in \mathbb C^{m \times m}$, $n \ge 1$, satisfying the condition
$$
\| K \|_{l_2} = \sqrt{\sum_{n = 1}^{\infty} \| K_n \|^2} < \infty.
$$
\item We use the same notation $\| . \|_{L_2}$ for the norms in $L_2(0,\pi)$, $L_2((0,\pi); \mathbb C^m)$, and $L_2((0,\pi); \mathbb C^{m \times m})$ and the same notation $\| . \|_{l_2}$ for the norms in $l_2(\mathbb C^{m \times m})$ and in the space $l_2$ of scalar sequences, since it does not cause ambiguity.
\item $\overline{j,k}$ means $j, j+1, \dots, k$.
\item On the set of integer pairs 
\begin{equation} \label{defJ}
J := \bigl\{ (n, k) \colon n \in \mathbb N, \, k = \overline{1,m} \bigr\},
\end{equation}
we define the order
$$
(n_1, k_1) < (n_2, k_2) \quad \stackrel{def}{\Leftrightarrow} \quad n_1 < n_2 \:\: \text{or} \:\: (n_1 = n_2 \:\: \text{and} \:\: k_1< k_2).
$$
\item For sequences of form $\{ \la_{nk} \}_{n \ge 1, \, k = \overline{1,m}}$, we use a shorter notation $\{ \la_u \}_J$, $u = (n,k)$.
\item Denote by $\mathcal L = \mathcal L(Q, h, H)$ the boundary value problem \eqref{eqv}--\eqref{bc}. Along with $\mathcal L$, we also consider the boundary value problems $\tilde{\mathcal L} = \mathcal L(\tilde Q, \tilde h, \tilde H)$ and $\mathcal L_p = \mathcal L(Q_p, h_p, H_p)$ ($p \ge 1$) of the same form but with different coefficients. If a symbol $\gamma$ denotes an object related to $\mathcal L$, then the symbols $\tilde \gamma$ and $\gamma_p$ will denote the analogous objects related to $\tilde{\mathcal L}$ and $\mathcal L_p$, respectively, and $\hat \gamma := \gamma - \tilde \gamma$.
If a notation has lower indices (e.g., $\la_{nk}$), then we use the upper index $^{(p)}$ (e.g., $\la_{nk}^{(p)}$). 
\item In the estimates, we use the notations $C(A_1, A_2, \dots)$ and $c(A_1, A_2, \dots)$ for various positive constants depending only on $A_1$, $A_2$, \dots (e.g., $C(R)$, $c(\Omega, \eps)$).
\item $\rho = \sqrt{\la}$, $\rho_{nk} = \sqrt{\la_{nk}}$, where the square root branch is chosen so that $\arg \rho \in (-\frac{\pi}{2}, \frac{\pi}{2}]$.
\end{itemize}

Let us consider the boundary value problem $\mathcal L$ of form \eqref{eqv}--\eqref{bc} under the additional restriction
\begin{equation} \label{omega0}
h + H + \frac{1}{2} \int_0^{\pi} Q(x) \, dx = 0_m,
\end{equation}
which allows us to simplify the technique and to make formulations of the main results clearer. Removing this restriction is discussed in Section~\ref{sec:gen}.

It is well-known that the matrix Sturm-Liouville problem \eqref{eqv}--\eqref{bc} has a countable set of real eigenvalues. Let us describe their properties due to \cite{Bond11}. It is convenient to number the eigenvalues as $\{ \la_{nk} \}_{n \ge 1, \, k = \overline{1,m}}$ according to their asymptotics 
\begin{equation} \label{asymptla}
\rho_{nk} := \sqrt{\la_{nk}} = n - 1 + \frac{\varkappa_{nk}}{n}, \quad \{ \varkappa_{nk} \} \in l_2, \quad n \ge 1, \quad k = \overline{1,m}.
\end{equation}
Without loss of generality, we assume that $\la_u \le \la_v$ if $u < v$, $u,v \in J$, where $J$ is defined by \eqref{defJ}. Note that eigenvalues can be multiple. Since the problem $\mathcal L$ is self-adjoint, algebraic multiplicities of eigenvalues are equal to the corresponding geometric multiplicities and do not exceed $m$. 
A number of occurrences of each eigenvalue in
the sequence $\{ \la_u \}_J$ equals its multiplicity.

Denote by $\varphi(x, \la)$ the $(m \times m)$-matrix solution of equation \eqref{eqv} satisfying the initial conditions
\begin{equation} \label{icvv}
\vv(0, \la) = I_m, \quad \vv'(0,\la) = h.
\end{equation}

The vector eigenfunctions of the boundary value problem $\mathcal L$ form an orthonormal basis $\{ Y_u(x) \}_J$ in $L_2((0,\pi); \mathbb C^m)$. Obviously, the eigenfunctions they can be represented as $Y_u(x) = \vv(x, \la_u) v_u$, where $v_u = Y_u(0)$ are vectors of $\mathbb C^m$. Let us call $\{ v_u \}_J$ \textit{norming vectors}. Clearly, the vectors $\{ v_u \}_J$ can be chosen non-uniquely. We work with any fixed sequence of norming vectors that corresponds to an orthonormal basis of eigenfunctions.

We call $\{ \la_u, v_u \}_J$ \textit{the spectral data} of the problem $\mathcal L$. Consider the following inverse spectral problem.

\begin{ip} \label{ip:1}
Given the spectral data $\{ \la_u, v_u \}_J$, find $Q$, $h$, and $H$.
\end{ip}

In the scalar case $m = 1$, the norming vectors turn into the norming constants 
$|v_n|^2 = \left( \int_0^{\pi} \vv^2(x, \la_n) \, dx \right)^{-1}$.
Thus, Inverse Problem~\ref{ip:1} generalizes to the matrix case the classical inverse problem of Marchenko \cite{Mar50}, Gelfand and Levitan \cite{GL51} by spectral function. For the case of a finite interval, the step-like spectral function is specified by the eigenvalues and the norming constants (or weight numbers).

\begin{example} \label{ex:zero}
Let $Q(x) \equiv 0_m$, $h = H = 0_m$. Then $\la_{nk} = (n-1)^2$, $n \ge 1$, $k = \overline{1,m}$. Thus, the spectrum of the matrix Sturm-Liouville operator coincides with the one of the scalar Sturm-Liouville problem $-y'' = \la y$, $x \in (0,\pi)$, with the Neumann boundary conditions $y'(0) = y'(\pi) = 0$, but each eigenvalue has multiplicity $m$. Furthermore, $\vv(x, \la) = \cos(\sqrt{\la}x) I_m$ and
$$
   v_{1k} = \sqrt{\frac{1}{\pi}} e_k, \quad v_{nk} = \sqrt{\frac{2}{\pi}} e_k, \:\: n \ge 2, \quad k = \overline{1,m},
$$
where $\{ e_k \}_{k = 1}^m$ is any orthonormal basis in $\mathbb C^m$.
\end{example}

The uniqueness theorem for Inverse Problem~\ref{ip:1} follows from the results of \cite{Yur06} (see Section~\ref{sec:uchar} for details):

\begin{thm} \label{thm:uniq}
The spectral data $\{ \la_u, v_u \}_J$ uniquely specify $Q$, $h$, and $H$.
\end{thm}

Next, basing on the results of the previous studies \cite{Bond11, Bond19}, we obtain Theorem~\ref{thm:char} below on the spectral data characterization for the problem $\mathcal L$.

Denote by $\mathcal P$ the set of the triples 
\begin{equation} \label{classP}
(Q, h, H) \in L_2((0,\pi); \mathbb C^{m \times m}) \times \mathbb C^{m \times m} \times \mathbb C^{m \times m}
\end{equation}
satisfying \eqref{sa} and \eqref{omega0}.

Denote by $\mathcal S$ the set of sequences $\{ \la_u, v_u \}_J$ such that $\la_u \in \mathbb R$, $\la_{u_1} \le \la_{u_2}$ for $u_1 < u_2$, and $v_u \in \mathbb C^m$, $u \in J$.

\begin{thm} \label{thm:char}
For a sequence $\{ \la_u, v_u \}_J \in \mathcal S$ to be the spectral data of the problem 
\eqref{eqv}--\eqref{bc} with $(Q, h, H) \in \mathcal P$, the following conditions are necessary and sufficient:
\begin{enumerate}
\item Asymptotics \eqref{asymptla} and
\begin{equation} \label{asymptV} 
V^*_n V_n =  \frac{2}{\pi} I_m + \frac{K_n}{n}, \quad \{ K_n \} \in l_2(\mathbb C^{m \times m}), \quad n \ge 1, 
\end{equation}
where $V_n$ is the $(m \times m)$ matrix composed of the vectors $\{ v_{nk} \}_{k = 1}^m$.
\item The sequence $\{ v_u \cos (\rho_u x) \}_J$ is complete in $L_2((0,\pi); \mathbb C^m)$.
\end{enumerate}
\end{thm}

\begin{remark} \label{rem:char1}
For $m = 1$, condition~2 of Theorem~\ref{thm:char} becomes trivial. Indeed, there holds $v_n \ne 0$, 
and the sequence $\{ \cos (\rho_n x) \}_{n \ge 1}$ with $\rho_n = n - 1 + O(n^{-1})$, $\arg \rho_n \in \left( -\tfrac{\pi}{2}, \tfrac{\pi}{2}\right]$, is complete in $L_2(0,\pi)$ if and only if the values $\{ \rho_n \}_{n \ge 1}$ are all distinct. Thus, in the scalar case, condition~2 can be equivalently replaced by
$$
\la_n \ne \la_k \quad (n \ne k), \quad v_n \ne 0, \quad n \ge 1.
$$
Then Theorem~\ref{thm:char} turns into the classical result of inverse Sturm-Liouville theory (see, e.g., \cite[Theorem~1.6.2]{FY01}). However, in the matrix case, condition~2 is non-trivial, which is shown by Example~\ref{ex:fail}.
\end{remark}

\begin{remark} \label{rem:Riesz}
The vector functions from condition~2 of Theorem~\ref{thm:char} are closely related to eigenfunctions. Indeed, the solution $\vv(x, \la)$ satisfies the standard asymptotics
\begin{equation} \label{asymptvv}
\vv(x, \rho^2) = \cos (\rho x) I_m + O\left( \rho^{-1} \exp(|\mbox{Im}\rho| x)\right), \quad |\rho| \to \infty,
\end{equation}
so
$$
\vv(x, \la_{nk}) v_{nk} = v_{nk} \cos(\rho_{nk} x) + O\left( n^{-1} \right), \quad n \to \infty.
$$
Thus, the sequence $\{ v_u \cos (\rho_u x) \}_J$ is quadratically close to the orthonormal basis of eigenfunctions $\{ Y_u \}_J$. Since by condition~2 of Theorem~\ref{thm:char} the sequence $\{ v_u \cos (\rho_u x) \}_J$ is complete, then it is a Riesz basis in $L_2((0,\pi); \mathbb C^m)$.
\end{remark}

Note that the boundary value problem \eqref{eqv}--\eqref{bc} can be equivalently represented in the form
\begin{gather} \label{eqsi}
-(Y^{[1]})' - \sigma(x) Y^{[1]} - \sigma^2(x) Y =\lambda Y, \quad x \in (0, \pi), \\ \label{bcsi}
Y^{[1]}(0) = 0, \quad Y^{[1]}(\pi) + \check{H} Y(\pi) =0,
\end{gather}
where
\begin{equation} \label{defsi}
\sigma(x) := h + \int_0^x Q(t) \, dt, \quad \check{H} := H +\sigma(\pi),
\end{equation}
and $Y^{[1]}(x) := Y'(x) - \sigma(x) Y(x)$ is the quasi-derivative. For $Q \in L_2((0,\pi); \mathbb C^{m \times m})$, we have $\sigma \in W_2^1([0,\pi]; \mathbb C^{m \times m})$.

The direct and inverse matrix Sturm-Liouville problems in a more general form than \eqref{eqsi}--\eqref{bcsi} with $\sigma \in L_2((0,\pi); \mathbb C^{m \times m})$ have been investigated in \cite{Bond21-mn, Bond21-amp}.

Let us provide the spectral data characterization for the problem \eqref{eqsi}--\eqref{bcsi} with $\sigma \in L_2((0,\pi); \mathbb C^{m \times m})$, which follows from the results of \cite{Bond21-mn, Bond21-amp} similarly to Theorem~\ref{thm:char}. Denote by $\varphi(x, \la)$ the $(m \times m)$ matrix solution of equation \eqref{eqsi} satisfying the initial conditions
\begin{equation} \label{icvvsi}
\varphi(0,\la) = I_m, \quad \varphi^{[1]}(0,\la) = 0_m.
\end{equation}
Obviously, if $\sigma \in W_2^1([0,\pi]; \mathbb C^{m \times m})$ and the relations \eqref{defsi} hold, then the initial conditions \eqref{icvvsi} coincide with \eqref{icvv}. The norming vectors $\{ v_u \}_J$ are defined for the problem \eqref{eqsi}--\eqref{bcsi} in the same way as described above.

\begin{thm} \label{thm:charsi}
For a sequence $\{ \la_u, v_u \}_J \in \mathcal S$ to be the spectral data of the problem \eqref{eqsi}--\eqref{bcsi} with $\sigma \in L_2((0,\pi); \mathbb C^{m \times m})$, $\sigma(x) = \sigma^*(x)$ a.e. on $(0,\pi)$, and $\check{H} = \check{H}^*$,
the following conditions are necessary and sufficient:
\begin{enumerate}
\item Asymptotic relations:
\begin{gather*}
\rho_{nk} := \sqrt{\la_{nk}} = n - 1 + \check{\varkappa}_{nk}, \quad \{ \check{\varkappa}_{nk} \} \in l_2, \quad n \ge 1, \: k =\overline{1,m}, \\
V_n^* V_n = \frac{2}{\pi} I_m + \check{K}_n, \quad \{ \check{K}_n \} \in l_2(\mathbb C^{m \times m}), \quad n \ge 1.
\end{gather*}
\item The sequence $\{ v_u \cos(\rho_u x) \}_J$ is complete in $L_2((0,\pi); \mathbb C^m)$.
\end{enumerate}
\end{thm} 

\section{Main results} \label{sec:main}

In this section, we formulate our main theorems on uniform bounds for the direct and inverse problems, and on the uniform stability of the inverse problem for the matrix Sturm-Liouville operator given by \eqref{eqv}--\eqref{bc}.

For $R > 0$, denote by $\mathcal P_R$ the set of the triples $(Q, h, H) \in \mathcal P$ satisfying 
$$
\| Q \|_{L_2} + \| h \| + \| H \| \le R.
$$

The following theorem provides uniform bounds for a solution of the direct spectral problem $(Q, h, H) \mapsto \{ \la_u, v_u \}_J$ on $\mathcal P_R$.

\begin{thm} \label{thm:bounddir}
Let $R > 0$ be fixed. Then the spectral data $\{ \la_u, v_u \}_J$ of the problem \eqref{eqv}--\eqref{bc} 
with parameters $(Q, h, H) \in \mathcal P_R$ fulfill the asymptotic formulas \eqref{asymptla} and \eqref{asymptV} with the remainder terms satisfying
\begin{equation} \label{boundabove}
\| \{ \varkappa_{nk} \} \|_{l_2} \le \Omega, \quad \| \{ K_n \} \|_{l_2(\mathbb C^{m \times m})} \le \Omega,
\end{equation}
where $\Omega > 0$ depends only on $R$. Moreover, $\{ v_u \cos(\rho_u x) \}_J$ 
is a Riesz basis in $L_2((0,\pi); \mathbb C^m)$ and, for any $\{ a_u \} \in l_2$, there holds
\begin{equation} \label{RBbound}
\left \| \sum_{u \in J} a_u v_u \cos(\rho_u x) \right\|_{L_2} \ge \eps \|\{ a_u \} \|_{l_2},
\end{equation}
where $\eps > 0$ depends only on $R$.
\end{thm}

\begin{cor}
(i) Theorem~\ref{thm:bounddir} immediately implies the uniform boundedness of the vectors norms:
\begin{equation} \label{boundv}
0 < c(R) \le \| v_u \| \le C(R), \quad u \in J.
\end{equation}
Indeed, the upper estimate in \eqref{boundv} follows from \eqref{asymptV} and \eqref{boundabove}, while the lower estimate follows from \eqref{RBbound}.

(ii) Under the hypothesis of Theorem~\ref{thm:bounddir}, the uniform upper bound for the Riesz basis also holds according to the asymptotics \eqref{asymptla}, \eqref{asymptV} and the estimates \eqref{boundabove}:
$$
\left \| \sum_{u \in J} a_u v_u \cos (\rho_u x) \right\|_{L_2} \le C(R) \|\{ a_u \} \|_{l_2}.
$$
\end{cor}

Let us define the set of spectral data satisfying the assertions of Theorem~\ref{thm:bounddir}.

\begin{defin}[Set $\mathcal S_{\Omega, \eps}$] \label{def:B}
For $\Omega > 0$ and $\eps > 0$, denote by $\mathcal S_{\Omega, \eps}$ the set of all the sequences 
$\{ \la_u, v_u \}_J$ in $\mathcal S$
such that the asymptotics \eqref{asymptla} and \eqref{asymptV} hold,
the sequence $\{ v_u \cos(\rho_u x) \}_J$ is a Riesz basis in $L_2((0,\pi); \mathbb C^m)$, 
and the estimates \eqref{boundabove} and \eqref{RBbound} are fulfilled with the initially fixed $\Omega$ and $\eps$.
\end{defin}

The following theorem establishes the uniform boundedness of Inverse Problem~\ref{ip:1} on $\mathcal S_{\Omega,\eps}$.

\begin{thm} \label{thm:boundinv}
Let $\Omega > 0$ and $\eps > 0$ be fixed. 
Then, any sequence $\{ \la_u, v_u \}_J \in \mathcal S_{\Omega, \eps}$ is the spectral data of the boundary value problem 
\eqref{eqv}--\eqref{bc} for a unique triple $(Q, h, H) \in\mathcal P$. Moreover, $(Q, h, H) \in \mathcal P_R$, where $R > 0$ depends only on $\Omega$ and $\eps$.
\end{thm}

Thus, Theorems~\ref{thm:bounddir} and~\ref{thm:boundinv} assert that:
\begin{itemize}
\item The direct spectral transform $(Q, h, H) \mapsto \{ \la_u, v_u \}_J$ maps $\mathcal P_R$ into $\mathcal S_{\Omega,\eps}$.  
\item The inverse spectral transform $\{ \la_u, v_u \}_J \mapsto (Q, h, H)$ maps $\mathcal S_{\Omega, \eps}$ into $\mathcal P_R$.
\end{itemize}

\begin{remark} \label{rem:sep}
For $m = 1$, the condition \eqref{RBbound}:
$$
\left \| \sum_{n=1}^{\infty} a_n v_n \cos(\rho_n x) \right\|_{L_2} \ge \eps \|\{ a_n \} \|_{l_2}
$$
can be replaced by the following equivalent conditions:
\begin{equation} \label{sepeig}
   \al_n \ge \eps_1, \quad \la_{n + 1} - \la_n \ge \eps_1, \quad n \ge 1, \quad \eps_1 > 0,
\end{equation}
where $\al_n = |v_n|^2$, $\la_n = \rho_n^2$. Note that 
uniform bounds for appropriate Riesz bases of cosines, sines, and exponentials follow from uniform separation of eigenvalues according to \cite{Hryn10}.
\end{remark}

Next, we will show that, on the set $\mathcal S_{\Omega,\eps}$, the uniform stability of Inverse Problem~\ref{ip:1} holds. For this purpose, we need some additional notations. Let us introduce the sequence $\{ \rho_u, \be_u \}_J$, $\rho_u := \sqrt{\la_u}$, $\be_u := v_u v_u^*$, and call this sequence the $(\rho,\be)$\textit{-spectral data} of the corresponding problem $\mathcal L$. Furthermore, let us define a partition of the index set \eqref{defJ}.

\begin{defin}[Partition $\mathscr P$] \label{def:Js}
Let $\mathscr P = \{ J_s \}_{s \ge 1}$ be a partition of the set $J$ into finite subsets
$$
J_s = \{ u \in J \colon u_s \le u < u_{s+1} \}, \quad (1,1) = u_1 < u_2 < \dots < u_s < u_{s+1} < \dots
$$

For two sequences of spectral data $\mathscr S := \{ \la_u, \be_u \}_J$ and $\tilde{\mathscr S} := \{ \tilde\la_u, \tilde \be_u \}_J$, consider the corresponding $(\rho,\be)$-spectral data $\{ \rho_u, \be_u \}_J$ and $\{ \tilde \rho_u, \tilde \be_u \}_J$, respectively.
Denote
\begin{gather}  \label{sumbe}
\be(J_s) := \sum_{u \in J_s} \be_u, \quad \tilde \be(J_s) := \sum_{u \in J_s} \tilde\be_u, \\ \label{defzeta}
\zeta(J_s) := \sum_{u \in J_s} |\rho_u - \tilde \rho_u| + \sum_{u, v \in J_s} \bigl( |\rho_u - \rho_v|
+ |\tilde \rho_u - \tilde \rho_v|\bigr) + \| \be(J_s) - \tilde \be(J_s) \|, \\ \label{defZ}
Z(\mathscr P, \mathscr S, \tilde{\mathscr S}) := \| \{ s \, \zeta(J_s) \}_{s \ge 1} \|_{l_2}.
\end{gather}
\end{defin}

Roughly speaking, the value $Z$ characterizes ``the difference'' between two spectral data sets $\mathscr S$ and $\tilde{\mathscr S}$ using the partition $\mathscr P$. The sum $\sum_{u,v \in J_s}$ in \eqref{defzeta} also involves the differences $|\rho_u - \rho_v|$ and $|\tilde \rho_u - \tilde \rho_v|$ between values related to the same spectral data set.

\begin{example}
The partition $\mathscr P \colon J_n = \{ (n,k) \}_{k = 1}^m$, $n \in \mathbb N$, satisfies Definition~\ref{def:Js} and implies $Z(\mathscr P, \mathscr S, \tilde{\mathscr S}) < \infty$ for any two spectral data sequences $\mathscr S$ and $\tilde{\mathscr S}$ of problems \eqref{eqv}--\eqref{bc} with coefficients in $\mathcal P$ in view of the asymptotics \eqref{asymptla} and \eqref{asymptV}.
\end{example}

For any possible choice of the partition $\mathscr P$, we get the following theorem on the uniform stability of Inverse Problem~\ref{ip:1}.

\begin{thm} \label{thm:unistab}
Let $\mathcal L = \mathcal L(Q, h, H)$ and $\tilde{\mathcal L} = \mathcal L(\tilde Q, \tilde h, \tilde H)$ be two problems of form \eqref{eqv}--\eqref{bc}, 
whose spectral data $\mathscr S = \{ \la_u, v_u \}_J$ and $\tilde{\mathscr S} = \{ \tilde \la_u, \tilde v_u \}_J$, respectively, lie in $\mathcal S_{\Omega, \eps}$. Then
\begin{equation} \label{uni}
\| Q - \tilde Q \|_{L_2}  + \| h - \tilde h \| + \| H - \tilde H \| \le C(\Omega, \eps) Z(\mathscr P, \mathscr S, \tilde{\mathscr S}),
\end{equation}
where $\mathscr P$ is any partition satisfying Definition~\ref{def:Js}.
\end{thm}

Note that $Z$ in \eqref{uni} depends on a non-unique choice of $\mathscr P$, norming vectors $\{ v_u \}_J$ and $\{ \tilde v_u \}_J$. Obviously, it is worth choosing them to make the value $Z$ as minimal as possible. 

For comparing Theorem~\ref{thm:unistab} with previous results, consider the following example. Other examples are provided in Section~\ref{sec:unistab}.

\begin{example} \label{ex:1}
Choose the partition $\mathscr P$ consisting of
the one-element subsets $J_{(n-1)m+k}:= \{ (n,k) \}$, $n \ge 1$, $k = \overline{1,m}$. Then Theorem~\ref{thm:unistab} implies the estimate \eqref{uni} with
\begin{equation} \label{Zex1}
Z = \left( \sum_{n = 1}^{\infty} \sum_{k = 1}^m \bigl( n \bigl( |\rho_{nk} - \tilde \rho_{nk}| + \| \be_{nk} - \tilde \be_{nk} \| \bigr) \bigr)^2\right)^{1/2}.
\end{equation}
The resulting estimate directly generalizes the uniform estimate (2.4) of \cite{Bond24} for the scalar Sturm-Liouville operator. A local version of this estimate in the scalar case is given in \cite[Theorem~1.6.4]{FY01}. However, in the matrix case, the value \eqref{Zex1} is not necessarily finite, which is shown by Example~\ref{ex:2}. Therefore, one needs to consider different partitions satisfying Definition~\ref{def:Js}.
\end{example}

For the scalar case $m = 1$, the uniform stability of the inverse Sturm-Liouville problems has been proved by Savchuk and Shkalikov \cite{SS10, SS13} for the Dirichlet boundary conditions and potentials $q \in W_2^{\theta}[0,\pi]$, $\theta > -1$. In \cite{SS10}, the inverse problem by two spectra was considered in more detail. The study \cite{SS13} is focused on the inverse problem by the spectral function. Our Theorems~\ref{thm:bounddir}, \ref{thm:boundinv}, and \ref{thm:unistab} generalize a part of results of \cite[Theorem~3.15]{SS13} to matrix-valued potentials of $L_2$, taking technical differences between the Dirichlet and Robin boundary conditions and Remark~\ref{rem:sep} into account.

\section{Uniqueness and characterization} \label{sec:uchar}

This section contains the proofs of Theorems~\ref{thm:uniq} and~\ref{thm:char} on the uniqueness and the spectral data characterization, respectively, for Inverse Problem~\ref{ip:1}. More precisely, we deduce Theorems~\ref{thm:uniq} and~\ref{thm:char} from  previous results of \cite{Yur06, Bond11, Bond19}, which were obtained in terms of the so-called weight matrices.

Denote by $\Phi(x, \la)$ the $(m \times m)$ matrix solution of equation \eqref{eqv} satisfying the boundary conditions
\begin{equation} \label{bcPhi}
\Phi'(0,\la) - h \Phi(0,\la) = I_m, \quad \Phi'(\pi,\la) + H \Phi(\pi,\la) = 0_m.
\end{equation}

The Weyl matrix is defined as $M(\la) := \Phi(0,\la)$. The elements of the matrix function $M(\la)$ are meromorphic in $\la$ and their poles belong to the eigenvalue set $\{ \la_u \}_J$. Introduce the weight matrices as the residues with respect to these poles:
$$
\al_u := \Res_{\la = \la_u} M(\la), \quad u \in J.
$$

\begin{prop}[\hspace*{-3pt}\cite{Yur06}] \label{prop:uniqw}
The spectral data $\{ \la_u, \al_u \}_J$ uniquely specify $Q$, $h$, and $H$.
\end{prop}

Note that the weight matrices in the sequence $\{ \al_u \}_J$ are repeated for multiple eigenvalues, and their ranks coincide with the corresponding eigenvalue multiplicities (see \cite{Bond11}). Let us introduce some notations to avoid repetitions. 

\begin{defin}[Matrices $\al_u'$ and vectors $\mathscr E_u$] \label{def:E}
Suppose that $\la_{u_1} = \la_{u_2} = \dots = \la_{u_r}$ ($1 \le r \le m$) is a group of multiple eigenvalues maximal by inclusion. Then $\al_{u_1} = \al_{u_2} = \dots = \al_{u_r}$ and $\rank \al_{u_1} = r$. 
Put $\al_{u_1}' := \al_{u_1}$ and $\al_{u_j}' := 0_m$ for $j = \overline{2,r}$.
Next, denote by $\{ \mathscr E_{u_j} \}_{j = 1}^r$ an orthonormal basis in the subspace $\mbox{Ran} \, \al_u$ of $\mathbb C^m$. Compose the sequences 
$\{ \al'_u \}_J$ and
$\{ \mathscr E_u \}_J$. 
\end{defin}

In terms of the introduced notations, the spectral data characterization is formulated as follows.

\begin{prop}[\hspace*{-3pt}\cite{Bond19}] \label{prop:charw}
For values $\{ \la_u, \al_u \}_J$ to be the eigenvalues and the weight matrices of a problem \eqref{eqv}--\eqref{bc} with $(Q, h, H) \in \mathcal P$, the following conditions are necessary and sufficient:
\begin{enumerate}
\item $\la_u \in \mathbb R$, $\la_u \le \la_w$ for $u < w$, and
\begin{equation} \label{structal}
\left.
\begin{array}{c}
\al_u = \al_u^* \ge 0, \quad \rank \al_u = \# \{ w \colon \la_w = \la_u \}, \\
\al_u = \al_w \:\: \text{if} \:\: \la_u = \la_w.
\end{array} \quad \right\}
\end{equation}
\item There hold the asymptotics \eqref{asymptla} and
\begin{equation} \label{asymptbe}
\be_n := \sum_{k = 1}^m \al'_{nk} = \frac{2}{\pi} I_m + \frac{\mathscr K_n}{n}, \quad \{ \mathscr K_n \} \in l_2(\mathbb C^{m \times m}), \quad n \ge 1.
\end{equation}
\item The sequence $\{ \mathscr E_u \cos(\rho_u x) \}_J$ is complete in $L_2((0,\pi); \mathbb C^m)$.
\end{enumerate}
\end{prop}

Let us establish the relation between norming vectors and weight matrices.

\begin{lem} \label{lem:relval}
Any sequence of norming vectors $\{ v_u \}_J$ is related to the weight matrices as follows:
\begin{equation} \label{findal}
\al_u = \sum_{w \colon \la_w = \la_u} v_w v_w^*, \quad u \in J.
\end{equation}
\end{lem}

\begin{proof}
According to the properties \eqref{structal}, each weight matrix can be expanded as
\begin{equation} \label{findalu}
\al_u = \sum_{w \colon \la_w = \la_u} g_w g_w^*,
\end{equation}
where $\{ g_w \}$ are non-zero mutually orthogonal column vectors.

Lemma~5 in \cite{Bond11} implies the relations
\begin{align} \label{al1}
& \al_u \int_0^{\pi} \vv^*(x, \la_u) \vv(x, \la_u) \, dx \, \al_u  = \al_u, \\ \label{al2}
& \al_w \int_0^{\pi} \vv^*(x, \la_w) \vv(x, \la_u) \, dx \, \al_u  = 0_m, \quad \la_w \ne \la_u.
\end{align}

Substituting the representation \eqref{findalu} into \eqref{al1} and \eqref{al2} and using the orthogonality of the vectors $\{ g_u \}$ corresponding to the same weight matrix, we obtain
\begin{align} \label{Evv1}
& g^*_u \int_0^{\pi} \vv^*(x, \la_u) \vv(x, \la_u) \, dx \, g_u = 1, \\ \label{Evv2}
& g^*_w \int_0^{\pi} \vv^*(x, \la_w) \vv(x, \la_u) \, dx \, g_u = 0, \quad w \ne u.
\end{align}

Furthermore, from the proof of \cite[Lemma~5]{Bond11}, we have 
$$
(\vv'(\pi, \la_u) + H \vv(\pi, \la_u)) \al_u = 0_m.
$$
This together with \eqref{findalu}, \eqref{Evv1}, and \eqref{Evv2} imply that $\{ \vv(x, \la_u) g_u \}_J$ are orthonormal eigenfunctions of the problem $\mathcal L$.

Let $\{ \vv(x, \la_u) v_u \}_J$ be another sequence of orthonormal eigenfunctions ($v_u \in \mathbb C^m$). For fixed $u$, denote by $G_u$ and $V_u$ the $(m \times r)$-matrices composed of the columns $\{ g_w \}$ and $\{ v_w \}$, respectively, for $w$ such that $\la_w = \la_u$. Here $r$ is the multiplicity of the eigenvalue $\la_u$. Then
$G_u = V_u S_u$, where $S_u$ is some unitary $(r \times r)$ matrix. Consequently, the relation \eqref{findalu} can be rewritten as
$$
\al_u = G_u G_u^* = V_u V_u^*,
$$
which is equivalent to \eqref{findal}.
\end{proof}

Proposition~\ref{prop:uniqw} and Lemma~\ref{lem:relval} immediately yield Theorem~\ref{thm:uniq}.

\begin{lem} \label{lem:asymptbe}
The asymptotics \eqref{asymptV} is equivalent to \eqref{asymptbe}. Moreover, the estimates $\| \{ K_n \} \|_{l_2} \le \Omega$ and $\| \{ \mathscr K_n \} \|_{l_2} \le \Omega$ are equivalent to each other.
\end{lem}

\begin{proof}
The definition of $\be_n$ in \eqref{asymptbe} and \eqref{findal} imply $\be_n = V_n V_n^*$, so \eqref{asymptbe} can be rewritten as
$$
V_n V_n^* = \frac{2}{\pi} I_m + \frac{\mathscr K_n}{n}, \quad \{ \mathscr K_n \} \in l_2(\mathbb C^{m \times m}), \quad n \ge 1.
$$
Then, the assertion of the lemma is obtained by applying 
the singular value decomposition of~$V_n$.
\end{proof}

\begin{proof}[Proof of Theorem~\ref{thm:char}]
In view of \eqref{findal}, the conditions of Theorem~\ref{thm:char} are equivalent to the ones of Proposition~\ref{prop:charw}. Indeed, the relations \eqref{structal} follow from \eqref{findal} and the linear independence of the vectors $\{ v_w \colon \la_w = \la_u \}$. The asymptotics \eqref{asymptV} and \eqref{asymptbe} are equivalent by Lemma~\ref{lem:asymptbe}.
Finally, it follows from \eqref{findal} and Definition~\ref{def:E} that
$$
\mbox{span} \, \{ v_w \colon \la_w = \la_u \} = \mbox{span} \, \{ \mathscr E_w \colon \la_w = \la_u \}.
$$
Consequently, the completeness conditions for the sequences $\{ v_u \cos(\rho_u x) \}_J$ and
$\{ \mathscr E_u \cos(\rho_u x) \}_J$ are equivalent to each other, which concludes the proof.
\end{proof}

For illustrating the non-triviality of condition~2 in Theorem~\ref{thm:char}, consider an example.

\begin{example} \label{ex:fail}
Let $m = 2$, $\la_{11} \ne \la_{12}$, $\la_{n1} = \la_{n2} = (n-1)^2$ for $n \ge 2$, and
$$
   v_{11} = v_{12} = \sqrt{\frac{1}{\pi}}\begin{bmatrix}
                              1 \\ 0
			\end{bmatrix}, \quad
   v_{n1} = \sqrt{\frac{2}{\pi}}\begin{bmatrix}
                              1 \\ 0
			\end{bmatrix}, \quad 
   v_{n2} = \sqrt{\frac{2}{\pi}}\begin{bmatrix}
                              0 \\ 1
			\end{bmatrix}, \quad n \ge 2.
$$
Since $\la_{nk}$ and $v_{nk}$ for $n \ge 2$ coincide with the spectral data of the zero problem of Example~\ref{ex:zero}, condition~1 of Theorem~\ref{thm:char} is fulfilled. The sequence from condition~2 takes the form
\begin{gather*}
   \chi_{11} = \sqrt{\frac{1}{\pi}} \begin{bmatrix} 1 \\ 0 \end{bmatrix} \cos (\rho_{11} x), \quad
   \chi_{12} = \sqrt{\frac{1}{\pi}} \begin{bmatrix} 1 \\ 0 \end{bmatrix} \cos (\rho_{12} x), \\
   \chi_{n1} = \sqrt{\frac{2}{\pi}} \begin{bmatrix} 1 \\ 0 \end{bmatrix} \cos ((n-1)x), \quad
   \chi_{n2} = \sqrt{\frac{2}{\pi}} \begin{bmatrix} 0 \\ 1 \end{bmatrix} \cos ((n-1)x), \quad n \ge 2.
\end{gather*}
Obviously, the vector $\begin{bmatrix} 0 \\ 1 \end{bmatrix}$ is orthogonal to all $\chi_{nk}$ in $L_2((0,\pi); \mathbb C^2)$, so condition~2 of Theorem~\ref{thm:char} is violated. Hence, the given values $\{ \la_u, v_u \}_J$ are not the spectral data of a problem $\mathcal L$ of form \eqref{eqv}--\eqref{bc}.
\end{example}

\section{Bounds for the direct problem} \label{sec:bounddir}

In this section, we present the proof of Theorem~\ref{thm:bounddir}. The most complicated part is the proof of the lower estimate \eqref{RBbound}. The main idea consists in the passage to the limit in the Sobolev space $W_2^{\al}$, $\frac{1}{2} < \al < 1$, and the application of the compact embedding $W_2^1[0,\pi] \subset W_2^{\al}[0,\pi]$. More information about the scale of the Sobolev spaces $W_2^{\al}$ with non-integer indices $\al$ and their usage in the inverse spectral theory for the Sturm-Liouville operators can be found in \cite{Trieb78, SS10}.

\begin{proof}[Proof of Theorem~\ref{thm:bounddir}]
Fix $R > 0$ and consider $(Q, h, H) \in \mathcal P_R$.
The eigenvalue asymptotics \eqref{asymptla} have been derived in \cite{Bond11}, and the uniform estimate $\| \{ \varkappa_{nk} \} \|_{l_2} \le C(R)$ follows from those arguments, so we omit the proof. The asymptotic relation \eqref{asymptV} follows from \eqref{asymptbe} and Lemma~\ref{lem:asymptbe}. The asymptotics \eqref{asymptbe} has been proved in \cite[Lemma~3]{Bond19} (see also \cite[Proposition~2.5]{CK09} for the case of the Dirichlet boundary conditions). Those proofs also readily imply the uniform estimate $\| \{ \mathscr K_n \} \|_{l_2} \le C(R)$, which yields $\| \{ K_n \} \|_{l_2} \le C(R)$ according to Lemma~\ref{lem:asymptbe}. 

Proceed to the proof of relation~\eqref{RBbound}.
According to Remark~\ref{rem:Riesz}, the sequence $\{ v_u \cos(\rho_u x) \}_J$ is a Riesz basis in $L_2((0,\pi); \mathbb C^m)$. Hence, the estimate \eqref{RBbound} is valid with $\eps = \eps(Q, h, H) > 0$. It remains to show that this bound is uniform on $\mathcal P_R$. The proof of the latter fact splits into several steps.

\underline{\textsc{Step 1}: Passing to the limit}.
Suppose that, on the contrary, there exists a sequence $\{ (Q_p, h_p, H_p) \}_{p \ge 1} \subset \mathcal P_R$ such that $\eps(Q_p, h_p, H_p)$ tends to zero as $p \to \infty$. 
Represent the corresponding matrix Sturm-Liouville problems in the form \eqref{eqsi}--\eqref{bcsi}, that is, define $\sigma_p$ and $\check{H}_p$ using $(Q_p, h_p, H_p)$ via \eqref{defsi}. Then, the pairs $(\sigma_p, \check{H}_p)$ ($p \ge 1$) belong to the ball
\begin{equation} \label{ballD}
\| \sigma \|_{W_2^1} + \| \check{H} \| \le D
\end{equation}
in the Banach space $\mathfrak B^1 := W_2^1([0,\pi]; \mathbb C^{m \times m}) \times \mathbb C^{m \times m}$, where the radius $D >0$ depends only on $R$.

Any ball in $W_2^1[0,\pi]$ is weakly compact, so the ball \eqref{ballD} is weakly compact in $\mathfrak B^1$. Therefore, there is a subsequence of $\{ (\sigma_p, \check H_p) \}_{p \ge 1}$ that weakly converges in $\mathfrak B^1$ to a limit $(\sigma, \check H)$ satisfying \eqref{ballD}. Without loss of generality, let $\{ (\sigma_p, \check H_p) \}_{p \ge 1}$ be such a subsequence.
The Sobolev space $W_2^1[0,\pi]$ is compactly embedded in $W_2^{\al}[0,\pi]$ for $\al < 1$ (see \cite{Trieb78, SS10}). Consequently, the sequence $\{ (\sigma_p, \check H_p) \}_{p \ge 1}$ strongly converges to $(\sigma, \check H)$ in $W_2^{\al}([0,\pi]; \mathbb C^{m \times m}) \times \mathbb C^{m \times m}$. 
Moreover, if we choose $\al > 1/2$, then the space $W_2^{\al}[0,\pi]$ is continuously embedded in the space $C[0,\pi]$ of continuous functions. Therefore, the convergence in $W_2^{\al}[0,\pi]$ implies the point-wise convergence. Consequently, $\sigma_p(0) \to \sigma(0)$ and $\sigma_p(\pi) \to \sigma(\pi)$ as $p \to \infty$. Recover $Q$, $h$, and $H$ from $\sigma$ and $\check{H}$ using formula \eqref{defsi}. Then, $h_p \to h$ and $H_p \to H$ as $p \to \infty$, so the relations
$$
h_p + H_p + \frac{1}{2} \int_0^{\pi} Q_p(x) \, dx = 0_m, \quad p \ge 1,
$$
imply \eqref{omega0}. Moreover, the conditions \eqref{sa} are fulfilled. Thus, $(Q, h, H) \in \mathcal P_R$.

Introduce the notations $\chi_u := v_u \cos(\rho_u x)$, $\chi_u^{(p)} := v_u^{(p)} \cos(\rho_u x)$, $u \in J$.
Let $a = \{ a_u \}_J$ be an arbitrary sequence of $l_2$.
Consider the series
\begin{equation} \label{sumser}
\sum_{u \in J} a_{nk} \chi_{nk}^{(p)} = \sum_{n,k \colon n \le N} a_{nk} \chi_{nk}^{(p)} + \sum_{n,k \colon n > N} a_{nk} (\chi_{nk}^{(p)} - \mathring{\chi}_{nk}^{(p)}) + \sum_{n,k \colon n > N} a_{nk} \mathring{\chi}_{nk}^{(p)},
\end{equation}
where $N > 0$ is some integer and $\mathring{\chi}_{nk}^{(p)}(x) := v_{nk}^{(p)} \cos ((n - 1) x)$. 

\smallskip

\underline{\textsc{Step 2:} Estimating terms for $n > N$}.
Using \eqref{asymptla}, \eqref{boundabove}, and the upper bound $\| v_{nk}^{(p)} \| \le C(R)$, we obtain the estimate
\begin{equation} \label{esta}
\left\| \sum_{n,k \colon n > N} a_{nk} (\chi_{nk}^{(p)} - \mathring{\chi}_{nk}^{(p)})\right\|_{L_2} \le \frac{C \| a \|_{l_2}}{N}, \quad p \ge 1.
\end{equation}
Next, represent the last series in \eqref{sumser} as follows:
$$
\sum_{n,k \colon n > N} a_{nk} \mathring{\chi}_{nk}^{(p)} = \sum_{n > N} V_n^{(p)} a_n \cos ((n-1)x) = \sum_{n > N} V_n b_n^{(p)} \cos ((n-1)x),
$$
where $a_n = [a_{nk}]_{k = 1}^m$ is the column vector and
\begin{equation} \label{defbnlarge}
b_n^{(p)} = [b_{nk}^{(p)}]_{k = 1}^m := V_n^{-1} V_n^{(p)} a_n.
\end{equation}
Using \eqref{asymptV} and \eqref{boundabove}, we get the relation
\begin{equation} \label{transab}
\| b_n^{(p)} \| = \| a_n \| (1 + \tau_n^{(p)}), \quad \| \{ n \tau_n^{(p)} \} \|_{l_2} \le C,
\end{equation}
where the estimate is uniform by $p \ge 1$.
Thus, we get
\begin{equation} \label{sumser1}
\sum_{n,k \colon n > N} a_{nk} \mathring{\chi}_{nk}^{(p)} = \sum_{n,k \colon n > N} b_{nk}^{(p)} \mathring{\chi}_{nk} = \sum_{n,k \colon n > N} b_{nk}^{(p)} (\mathring{\chi}_{nk} - \chi_{nk}) +  \sum_{n,k \colon n > N} b_{nk}^{(p)} \chi_{nk},
\end{equation}
where $\mathring{\chi}_{nk} := v_{nk} \cos((n-1)x)$. Similarly to \eqref{esta}, there holds
\begin{equation} \label{estb}
\left\| \sum_{n,k \colon n > N} b_{nk}^{(p)} (\mathring{\chi}_{nk} - \chi_{nk})\right\|_{L_2} \le \frac{C \| \{ b_{nk}^{(p)} \}_{n > N} \|_{l_2}}{N}, \quad p \ge 1.
\end{equation}
In view of \eqref{transab}, the vector $\{ b_{nk}^{(p)} \}$ in the right-hand side of \eqref{estb} can be replaced by $a$.

Let $\eta > 0$ be fixed. Choose an index $N$ (depending only on $R$ and $\eta$) such that
\begin{gather} \label{eta1}
\left\| \sum_{n,k \colon n > N} a_{nk} (\chi_{nk}^{(p)} - \mathring{\chi}_{nk}^{(p)})\right\|_{L_2} \le \frac{\eta}{3} \| a \|_{l_2}, \quad
\left\| \sum_{n,k \colon n > N} b_{nk}^{(p)} (\mathring{\chi}_{nk} - \chi_{nk})\right\|_{L_2} \le \frac{\eta}{3} \| a \|_{l_2}, \\ \label{eta2}
\| \{ b_{nk}^{(p)} \}_{n > N} \|_{l_2} \ge  (1 - \eta) \| \{ a_{nk} \}_{n > N} \|_{l_2},
\end{gather}
for every $p \ge 1$, which is possible due to the estimates \eqref{esta}, \eqref{transab}, and \eqref{estb}.

\smallskip

\underline{\textsc{Step 3:} Estimating terms for $n \le N$.}
Extract a subsequence of the considered sequence $\{(Q_p, h_p, H_p)\}_{p \ge 1}$ such that $v_{nk}^{(p)} \to v_{nk}^{\diamond}$ as $p \to \infty$ for $n \le N$ and $k = \overline{1,m}$. 

By virtue of \cite[Lemma 6.7]{Bond21-amp} and \eqref{findal}, we have
\begin{equation} \label{limp}
\lim_{p \to \infty} \la_u^{(p)} = \la_u, \quad
\lim_{p \to \infty} \sum_{w \colon \la_w = \la_u} v^{(p)}_w (v_w^{(p)})^* = \al_u.
\end{equation}
for each fixed $u \in J$. 
Hence 
$$
\lim_{p \to \infty}\chi_{nk}^{(p)} = \chi_{nk}^{\diamond} :=  v_{nk}^{\diamond} \cos \rho_{nk} x,
\quad n \le N, \: k = \overline{1,m}. 
$$

Represent the first sum in the right-hand side of \eqref{sumser} in the form
\begin{gather} \label{sumser2}
\sum_{n,k \colon n \le N} a_{nk} \chi_{nk}^{(p)} = \Sigma_1 + \Sigma_2, \quad 
\Sigma_1 := \sum_{n,k \colon n \le N} a_{nk} (\chi_{nk}^{(p)} - \chi_{nk}^{\diamond}), \quad \Sigma_2 := \sum_{n,k \colon n \le N} a_{nk} v_{nk}^{\diamond} \cos \rho_{nk} x.
\end{gather}

Obviously, the sum $\Sigma_1$ tends to zero as $p \to \infty$. Choose $p$ so large that 
\begin{equation} \label{estSi}
\| \Sigma_1 \|_{L_2} \le \frac{\eta}{3} \| a \|_{l_2}.
\end{equation}

It remains to investigate $\Sigma_2$.
Fix $u = (n, k)$, $n \le N$. Denote 
$$
\mathscr J_u := \{ w \colon \la_w = \la_u \}.
$$
The second relation in \eqref{limp} implies
$$
\sum_{w \in \mathscr J_u} v_w^{\diamond} (v_w^{\diamond})^* = \al_u.
$$
Due to Proposition~\ref{prop:charw}, $\rank \al_u$ coincides with the multiplicity of the eigenvalue $\la_u$, which equals cardinality of the set $\mathscr J_u$.
Therefore, the vectors $\{ v_w^{\diamond}\}_{\mathscr J_u}$ are linearly independent. Taking \eqref{findal} into account, we have
$$
\mbox{Ran} \, \al_u = \mbox{span} \{ v_w^{\diamond} \}_{\mathscr J_u} = \mbox{span} \{ v_w \}_{\mathscr J_u}. 
$$
Consequently, for each vector $\{ a_w \}_{\mathscr J_u}$, there exists a vector $\{ b_w \}_{\mathscr J_u}$ such that
\begin{equation} \label{combab}
\sum_{w \in \mathscr J_u} a_w v_w^{\diamond} = \sum_{w \in \mathscr J_u} b_w v_w, 
\end{equation}
moreover, 
\begin{equation} \label{bondb1}
\| \{ b_w \}_{\mathscr J_u} \| \ge c \| \{a_w \}_{\mathscr J_u} \|, \quad c > 0.
\end{equation}
Thus
\begin{equation} \label{sumser5}
\Sigma_2 = \sum_{n,k \colon n \le N} b_{nk} v_{nk} \cos \rho_{nk} x = \sum_{n,k \colon n \le N} b_{nk} \chi_{nk},
\end{equation}
where 
\begin{equation} \label{bondb2}
\| \{ b_{nk} \}_{n \le N} \|_{l_2} \ge c \| \{ a_{nk} \}_{n \le N} \|_{l_2}, \quad c > 0.
\end{equation}

Let us show that the constant $c$ in  \eqref{bondb2} does not depend on $N$. Indeed, according to \eqref{asymptV}, the sequence $\{ V_n^{(p)} \}_{n \ge 1}$ for each $p \ge 1$ satisfy the estimate
$$
\left\| (V_n^{(p)})^* V_n^{(p)} - \frac{2}{\pi} I_m \right\| \le \frac{C(R)}{n}, \quad n \ge 1,
$$
and so do $\{ V_n^{\diamond} \}_{n \ge 1}$. Roughly speaking, for sufficiently large $n$, the vectors $\left\{ \sqrt{\frac{\pi}{2}}v_{nk}^{(p)} \right\}_{k = 1}^m$ and $\left\{ \sqrt{\frac{\pi}{2}}v_{nk}^{\diamond} \right\}_{k = 1}^m$ are uniformly ``nearly orthonormal''. In view of \eqref{combab}, this implies
$$
\| \{ b_{ls} \}_{\mathscr J_u} \| = \bigl(1 + O(l^{-1})\bigr) \| \{ a_{ls} \}_{\mathscr J_u} \|,
$$
so the estimate \eqref{bondb2} holds with the same $c > 0$ for all $\mathscr J_u$, $u = (n,k)$ with sufficiently large~$n$.

\smallskip

\textsc{\underline{Step 4.}}
Combining \eqref{sumser}, \eqref{sumser1}, \eqref{sumser2}, and \eqref{sumser5}, we arrive at the relation
\begin{align} \nonumber
\sum_{n,k} a_{nk} \chi_{nk}^{(p)} = & \sum_{n,k} b_{nk} \chi_{nk} + \sum_{n,k \colon n \le N} a_{nk} (\chi_{nk}^{(p)} - \chi_{nk}^{\diamond}) \\ \label{sumser6} & + \sum_{n,k \colon n > N} a_{nk} (\chi_{nk}^{(p)} -\mathring{\chi}_{nk}^{(p)}) + \sum_{n,k \colon n > N} b_{nk} (\mathring{\chi}_{nk} - \chi_{nk})
\end{align}
where the numbers $b_{nk}$ ($k = \overline{1,m}$) are defined by \eqref{combab} for $n \le N$ and $b_{nk} := b_{nk}^{(p)}$ are given by \eqref{defbnlarge} for $n > N$.

Using \eqref{eta1}, \eqref{estSi}, and \eqref{sumser6}, we conclude that
\begin{equation} \label{RB1}
\left\| \sum_{u \in J} a_u \chi_u^{(p)} \right\|_{L_2} \ge \left\| \sum_{u \in J} b_u \chi_u \right\|_{L_2} - \eta \| a \|_{l_2}.
\end{equation}
The inequality \eqref{RBbound} for $(Q, h, H)$ implies
\begin{equation} \label{RB2}
\left\| \sum_{u \in J} b_u \chi_u \right\|_{L_2} \ge \eps \| \{ b_u \} \|_{l_2}, \quad \eps = \eps(Q, h,H) > 0.
\end{equation}

Combining \eqref{eta2}, \eqref{bondb2}, \eqref{RB1}, and \eqref{RB2}, we arrive at the estimate
$$
\left\| \sum_{u \in J} a_u \chi_u^{(p)} \right\|_{L_2} \ge (\eps \min \{ 1 - \eta, c \} -\eta) \| a \|_{l_2}.
$$
Since $\eta > 0$ can be chosen arbitrarily small, the latter estimate implies that the constants $\eps(Q_p, h_p, H_p)$ are bounded from below by a positive number as $p \to \infty$. This concludes the proof.
\end{proof}

\section{Main equation} \label{sec:maineq}

In this section, we derive the main equation of Inverse Problem~\ref{ip:1}, relying on the ideas of the method of spectral mappings \cite{Yur02}. For this purpose, we consider a problem $\mathcal L = \mathcal L(Q, h, H)$ with $(Q, h, H) \in\mathcal P$ and the model problem $\tilde{\mathcal L} = \mathcal L(0_m, 0_m, 0_m)$. The relations for the corresponding solutions $\vv(x, \la)$ and $\tilde \vv(x, \la)$ at the points of the two spectra are transformed to a linear equation in a suitable Banach space of infinite matrix sequences. 
Our construction of the main equation differs from the ones in the previous studies \cite{Yur06, Bond11, Bond19, Bond21-amp}. Specifically, we apply a modification that makes the operator in the main equation to be continuous w.r.t. the spectral data. For the scalar case ($m = 1$), this modification was introduced in \cite{Bond24}. Instead of the norming vectors $v_u$, it will be convenient for us to use the $\be_u := v_u v_u^*$, $u \in J$. In view of \eqref{asymptbe} and \eqref{findal}, there holds
$$
\be_n = \sum_{k = 1}^m \be_{nk}, \quad n \ge 1.
$$

Due to Example~\ref{ex:zero}, the spectral data of the model problem $\tilde{\mathcal L}$ satisfy the relations
\begin{gather*}
\sqrt{\tilde \la_{nk}} = \tilde \rho_{nk} = n-1, \quad n \ge 1, \quad k = \overline{1,m}, \\ 
\tilde \be_n = \sum_{k = 1}^m \tilde \be_{nk} = 
\begin{cases}
\frac{1}{\pi} I_m, & n = 1, \\
\frac{2}{\pi} I_m, & n \ge 2.
\end{cases}
\end{gather*}

For brevity, denote
$$
\tilde \rho_n := \tilde \rho_{nk}, \quad \tilde \la_n := \tilde \la_{nk}, \quad \hat \rho_{nk} := \rho_{nk} - \tilde \rho_n, \quad n \ge 1, \: k = \overline{1,m}.
$$

Obviously, we have
\begin{equation} \label{defvvt}
\tilde \vv(x, \rho^2) = \cos \rho x I_m.
\end{equation}
Introduce the matrix functions
\begin{gather} \label{defDt}
\tilde D(x, \mu, \la) := \int_0^x \tilde \vv(t, \mu) \tilde \vv(t, \la) \, dt, \\  \label{defwn}
w_n(x, \rho) := \frac{\vv(x, \rho^2) - \vv(x, \tilde \rho_n^2)}{\rho - \tilde \rho_n}, \quad \tilde w_n(x, \rho) := \frac{\tilde \vv(x, \rho^2) - \tilde \vv(x, \tilde \rho_n^2)}{\rho - \tilde \rho_n}, \\ \label{defWnt}
\tilde W_n(x, \theta, \rho) := \frac{\tilde D(x, \theta^2, \rho^2) - \tilde D(x, \theta^2, \tilde \rho_n^2)}{\rho - \tilde \rho_n}, \quad n \ge 1.
\end{gather}

The function $\tilde D(x, \mu, \la)$ is analytic in $\mu$ and $\la$, the functions $w_n(x, \rho)$ and $\tilde w_n(x, \rho)$, in $\rho$, and $\tilde W_n(x, \theta, \rho)$, in $\theta$ and $\rho$ for each fixed $x \in [0,\pi]$.
In particular,
$$
w_n(x, \tilde \rho_n) = \frac{d}{d\rho} \vv(x, \rho^2)_{|\rho = \tilde \rho_n}, \quad
\tilde W_n(x, \theta, \tilde \rho_n) = \frac{d}{d\rho} \tilde D(x, \theta^2, \rho^2)_{|\rho = \tilde \rho_n}.
$$

We rely on the following proposition, which has been obtained by the contour integration of spectral mappings in the complex plane of the spectral parameter $\la$.

\begin{prop}[\hspace*{-3pt}\cite{Yur06}]
The following relation holds
\begin{equation} \label{relvv}
\tilde \vv(x, \la) = \vv(x, \la) + \sum_{l = 1}^{\infty} \left( \sum_{s = 1}^m \vv(x, \la_{ls}) \be_{ls} \tilde D(x, \la_{ls}, \la) - \vv(x, \tilde \la_l) \tilde \be_l \tilde D(x, \tilde \la_l, \la) \right),
\end{equation}
where the series converges with brackets absolutely and uniformly by $x \in [0,\pi]$ and $\la$ on compact sets.
\end{prop}

Putting $\la = \la_{nk}$ and $\la = \tilde \la_n$ in \eqref{relvv}, we arrive at the infinite system of linear equations with respect to $\{ \vv(x, \la_{nk}), \vv(x, \tilde \la_n) \}$:
\begin{align}  \label{relvv1}
\tilde \vv(x, \la_{nk}) & = \vv(x, \la_{nk}) + \sum_{l = 1}^{\infty} \left( \sum_{s = 1}^m \vv(x, \la_{ls}) \be_{ls} \tilde D(x, \la_{ls}, \la_{nk}) - \vv(x, \tilde \la_l) \tilde \be_l \tilde D(x, \tilde \la_l, \la_{nk}) \right), \\ \label{relvv2}
\tilde \vv(x, \tilde \la_{n}) & = \vv(x, \tilde \la_{n}) + \sum_{l = 1}^{\infty} \left( \sum_{s = 1}^m \vv(x, \la_{ls}) \be_{ls} \tilde D(x, \la_{ls}, \tilde \la_{n}) - \vv(x, \tilde \la_l) \tilde \be_l \tilde D(x, \tilde \la_l, \tilde \la_{n}) \right), 
\end{align}
where $n \ge 1$, $k =\overline{1,m}$. The linear system \eqref{relvv1}--\eqref{relvv2} can be used for solving Inverse Problem~\ref{ip:1}. However, the series in \eqref{relvv1} and \eqref{relvv2} converge absolutely only ``with brackets'', so it is inconvenient to use them for our further analysis. Therefore, below we deduce from \eqref{relvv} the linear system with absolutely convergent series.

Define the matrix functions
\begin{equation} \label{defpsi}
\def\arraystretch{1.7}
\left.
\begin{array}{l}
\psi_{nk}(x) := w_n(x, \rho_{nk}), \quad \tilde \psi_{nk}(x) := \tilde w_n(x, \rho_{nk}), \quad k = \overline{1,m}, \\
\psi_{n,m+1}(x) := \vv(x, \tilde \la_n), \quad \tilde \psi_{n,m+1}(x) := \tilde \vv(x, \tilde \la_n), \quad n \ge 1.
\end{array}\qquad \right\}
\end{equation}

Using \eqref{relvv}, we obtain the following infinite system of equations with respect to $\{ \psi_{nk}(x) \}$:
\begin{equation} \label{sumpsi}
\tilde \psi_{nk}(x) = \psi_{nk}(x) + \sum_{l = 1}^{\infty} \sum_{s = 1}^{m + 1} \psi_{ls}(x) \tilde{\mathcal R}_{ls,nk}(x), \quad n \ge 1, \: k = \overline{1,m+1},
\end{equation}
where the matrix functions $\tilde{\mathcal R}_{ls,nk}(x)$ are defined as follows:
\begin{align} \label{defR}
\tilde{\mathcal R}_{ls,nk}(x) & := \hat \rho_{ls} \be_{ls} \tilde W_n(x, \rho_{ls}, \rho_{nk}), \quad s,k = \overline{1,m}, \\ \nonumber
\tilde{\mathcal R}_{l,m+1, nk}(x) & := \sum_{s = 1}^m \be_{ls} \tilde W_n(x, \rho_{ls}, \rho_{nk}) - \tilde \be_l \tilde W_n(x, \tilde \rho_l, \rho_{nk}), \quad k = \overline{1,m}, \\ \nonumber
\tilde{\mathcal R}_{ls, n, m+1}(x) & := \hat \rho_{ls} \be_{ls} \tilde D(x, \la_{ls}, \tilde \la_n), \quad s= \overline{1,m}, \\ \nonumber
\tilde{\mathcal R}_{l,m+1, n, m+1}(x) & := \sum_{s = 1}^m \be_{ls} \tilde D(x, \tilde \la_{ls}, \tilde \la_n) - \tilde \be_l \tilde D(x, \tilde \la_l, \tilde \la_n).
\end{align}

Let us analyze the convergence of the series \eqref{sumpsi}.
To characterize ``the difference'' between the spectral data of the problems $\mathcal L$ and $\tilde{\mathcal L}$, we introduce the quantities
\begin{equation} \label{defxi}
\xi_n := \sum_{k = 1}^m |\rho_{nk} - \tilde \rho_n| + \| \be_n - \tilde \be_n \|, \quad n \ge 1.
\end{equation}

By virtue of the asymptotics \eqref{asymptla} and \eqref{asymptbe}, the sequence $\{ n \xi_n \}_{n \ge 1}$ belongs to $l_2$. 
By the standard methods (see \cite[Section~1.3.1]{Yur02}), we obtain the estimates
\begin{gather} \label{estpsi}
\| \psi_{nk}(x) \| \le C, \\ \label{estpsiRt}
\| \tilde \psi_{nk}(x) \| \le C, \quad 
\| \tilde{\mathcal R}_{ls,nk}(x) \| \le \frac{C \xi_l}{|n-l| + 1},
\end{gather}
for $n,l \ge 1$, $k,s = \overline{1,m+1}$, $x \in [0,\pi]$. Moreover, if $\{ \la_{nk} \}$ and $\{ \be_n \}$ satisfy the asymptotics \eqref{asymptla} and \eqref{asymptbe}, respectively, and the corresponding remainders fulfill the uniform estimates
\begin{equation} \label{estOm}
\| \{ \varkappa_{nk} \} \|_{l_2} \le \Omega, \quad \| \{ \mathscr K_n \} \|_{l_2} \le \Omega,
\end{equation}
then \eqref{estpsiRt} holds with a constant $C = C(\Omega)$.

The estimates \eqref{estpsi} and \eqref{estpsiRt} imply that the series in \eqref{sumpsi} converges absolutely and uniformly by $x \in [0,\pi]$. Thus, the system \eqref{sumpsi} can be represented as a linear equation in a suitable Banach space.

Denote by $B$ the Banach space of bounded infinite sequences $a = \{ a_{nk} \}_{n \ge 1, \, k = \overline{1,m+1}}$, $a_{nk} \in \mathbb C^{m \times m}$, with the norm $\| a \|_B = \sup\limits_{n,k} \| a_{nk} \|$. For each fixed $x \in [0,\pi]$, introduce the linear operator $\tilde{\mathcal R}(x) \colon B \to B$ acting on an element $a \in B$ by the following rule:
\begin{equation} \label{opR}
(a \tilde{\mathcal R}(x))_{nk} := \sum_{l = 1}^{\infty} \sum_{s = 1}^{m+1} a_{ls} \tilde{\mathcal R}_{ls,nk}(x), \quad n \ge 1, \:k = \overline{1,m+1}.
\end{equation}

Thus, the action of the operator $\tilde{\mathcal R}(x)$ is the multiplication of an infinite $(m \times m)$-block row vector $a$ by the infinite $(m \times m)$-block matrix $[\tilde{\mathcal R}_{ls,nk}(x)]$. Following the notations of previous studies \cite{Yur06, Bond11, Bond19, Bond21-amp}, we write the operator $\tilde{\mathcal R}(x)$ to the right of operands to keep the correct order of non-commutative matrix multiplication.

It follows from \eqref{estpsi} and \eqref{estpsiRt} that $\psi(x) = \{ \psi_{nk}(x) \}_{n \ge 1, \, k = \overline{1,m+1}}$ and $\tilde \psi(x) = \{ \tilde \psi_{nk}(x) \}_{n \ge 1, \, k = \overline{1,m+1}}$ are elements of $B$ and the operator $\tilde{\mathcal R}(x)$ is bounded for each fixed $x \in [0,\pi]$:
$$
\| \tilde{\mathcal R}(x) \|_{B \to B} = \sup_{n,k} \sum_{l,s} \| \tilde{\mathcal R}_{ls,nk}(x) \| \le C \sup_{n\ge1} \sum_{l = 1}^{\infty} \frac{\xi_l}{|n-l|+1} < \infty.
$$

Denote by $I$ the identity operator in $B$. The reduction of Inverse Problem~\ref{ip:1} to a linear equation in $B$ is summarized in the following theorem.

\begin{thm} \label{thm:maineq}
Suppose that $\mathcal L = \mathcal L(Q, h, H)$, $(Q, h, H) \in \mathcal P$, $\tilde{\mathcal L} = \mathcal L(0_m,0_m,0_m)$, $\psi(x)$, $\tilde \psi(x)$, and $\tilde{\mathcal R}(x)$ are constructed as described above. Then, for each fixed $x \in [0,\pi]$, the vector $\psi(x) \in B$ satisfies the relation
\begin{equation} \label{main}
\tilde \psi(x) = \psi(x)(I + \tilde{\mathcal R}(x))
\end{equation}
in the Banach space $B$. Moreover, the operator $(I + \tilde{\mathcal R}(x))$ has a bounded inverse on $B$, so equation \eqref{main} is uniquely solvable.
\end{thm}

\begin{proof}
The relation \eqref{main} follows from \eqref{sumpsi} and \eqref{opR}. The invertibility of the operator $(I + \tilde{\mathcal R}(x))$ is proved similarly to \cite[Theorem~2]{Bond11}.
\end{proof}

The relation \eqref{main} is called \textit{the main equation} of Inverse Problem~\ref{ip:1}. Note that $\tilde \psi(x)$ and $\tilde{\mathcal R}(x)$ are constructed by using only the model problem $\tilde {\mathcal L}$ and the spectral data $\{ \la_u, \be_u \}_J$ of the boundary value problem $\mathcal L$. Using the solution $\psi(x)$ of the main equation \eqref{main}, one can find the matrix functions
\begin{equation} \label{recvv}
\vv(x, \la_{nk}) = \psi_{n,m+1}(x) + \hat \rho_{nk} \psi_{nk}(x), \quad k = \overline{1,m}, \quad \vv(x,\tilde \la_n) = \psi_{n,m+1}(x)
\end{equation}
and use them to recover $Q(x)$, $h$ and $H$.

\begin{remark}
Note that our construction differs from the ones provided in \cite{Bond11, Bond19}, since $\psi_{nk}(x)$ \eqref{defpsi} equals the derivative w.r.t. $\rho$ in the case $\rho_{nk} = \tilde \rho_n$. This modification of the classical scheme of the method of spectral mappings (see \cite{Yur02}) is important for the continuity of $\psi_{nk}(x)$, $\tilde \psi_{nk}(x)$, and $\tilde {\mathcal R}_{nk}(x)$ w.r.t. the spectral data (see \cite{Bond24}), which is important for investigating the stability. However, equation \eqref{main}, as well as the main equations in \cite{Bond11, Bond19}, is equivalent to the system \eqref{relvv1}--\eqref{relvv2} with respect to $\{ \vv(x, \la_{nk}), \vv(x,\tilde\la_n)\}$. Furthermore, due to \eqref{recvv}, the functions $\vv(x,\la_{nk})$ do not depend on $\psi_{nk}(x) = \dfrac{d}{d\rho} \vv(x,\rho^2)\big|_{\rho = \tilde\rho_n}$ in the case $\rho_{nk} = \tilde \rho_n$. For these reasons, the statement about the unique solvability of the main equation is transferred to the new construction without changes.
\end{remark}

\begin{remark} \label{rem:maineq}
According to \cite{Bond21-amp}, the assertion of Theorem~\ref{thm:maineq} is valid for the problem \eqref{eqsi}--\eqref{bcsi} with $\sigma \in L_2((0,\pi); \mathbb C^{m \times m})$. In this case, the matrix function $\vv(x, \la)$ is defined as the matrix solution of equation \eqref{eqsi} under the initial conditions \eqref{icvvsi}, and the other definitions remain the same as in this section.
\end{remark}

The matrix parameters $(Q, h, H)$ of $\mathcal L$ can be recovered from the solutions $\{ \varphi(x,\la_u),  \varphi(x, \tilde \la_u) \}_J$ by using the following proposition. For further needs, we formulate it in a more general form then the previous assertions of this section, assuming that the model problem $\tilde{\mathcal L}$ not necessarily equals $\mathcal L(0_m,0_m,0_m)$.

\begin{prop}[\hspace*{-3pt}\cite{Bond11}] \label{prop:relQhH}
Suppose that $(Q, h, H)$ and $(\tilde Q, \tilde h, \tilde H)$ belong to $\mathcal P$. Then
\begin{equation} \label{relQhH}
Q(x) = \tilde Q(x) -2 E_0'(x), \quad h = \tilde h - E_0(0), \quad
H = \tilde H + E_0(\pi),
\end{equation}
where
\begin{equation} \label{defE0}
E_0(x) := \sum_{u \in J} \bigl( \vv(x, \la_u) \be_u \tilde \vv^*(x, \la_u) - \vv(x, \tilde \la_u) \tilde \be_u \tilde \vv^*(x, \tilde \la_u) \bigr).
\end{equation}
The series in \eqref{defE0} converges in $W_2^1([0,\pi]; \mathbb C^{m \times m})$.
\end{prop}

\section{Bounds for the inverse problem} \label{sec:boundinv}

This section contains the proof of Theorem~\ref{thm:boundinv}. First, we show that, for $\{ \la_u, v_u \}_J \in \mathcal S_{\Omega,\eps}$, the corresponding inverse operator $(I + \tilde{\mathcal R}(x))^{-1}$ is uniformly bounded for fixed $\Omega > 0$ and $\eps > 0$. Our method involves constructing a weak limit $\{ \la_u, v_u \}_J$ of a sequence $\{ \la_u^{(p)}, v_u^{(p)} \}_J$ in $\mathcal S_{\Omega, \eps}$ as $p \to \infty$. The limit may fall out of the class of spectral data corresponding to potential matrices $Q(x)$ of $L_2((0,\pi); \mathbb C^{m \times m})$. Note that the compact embedding argument (see, e.g., \cite{SS10}) does not work here, since we have no asymptotics for separate vectors $v_{nk}$ as $n \to \infty$. Therefore, we develop an alternative approach. Using the asymptotics and the estimates given by Definition~\ref{def:B}, we prove that $\{ \la_u, v_u \}_J$ are the spectral data of a problem \eqref{eqsi}--\eqref{bcsi} for $\sigma \in L_2((0,\pi); \mathbb C^{m \times m})$, which actually corresponds to $Q \in W_2^{-1}((0,\pi); \mathbb C^{m \times m})$. Second, it is shown that the uniform boundedness of the inverse operator $(I + \tilde{\mathcal R}(x))^{-1}$ together with the other requirements imply that $Q$, $h$, and $H$ are uniformly bounded. The proof relies on the main equation \eqref{main}, on the construction of the infinite vector $\tilde \psi(x)$ and the operator $\tilde{\mathcal R}(x)$ in Section~\ref{sec:maineq}, and on the reconstruction formulas of Proposition~\ref{prop:relQhH}.

Let $\{ \la_u, v_u \}_J$ be any sequence of $\mathcal S_{\Omega, \eps}$. According to Definition~\ref{def:B}, this sequence satisfies the conditions of Theorem~\ref{thm:char}. Therefore, $\{ \la_u, v_u \}_J$ are the spectral data of some problem $\mathcal L$ with $(Q, h, H) \in \mathcal P$.
Using $\{ \la_u, v_u \}_J$ and the model problem $\tilde{\mathcal L} = \mathcal L(0_m,0_m,0_m)$, one can construct the bounded linear operator $\tilde{\mathcal R}(x) \colon B \to B$ as described in Section~\ref{sec:maineq}. 
By Theorem~\ref{thm:maineq}, there exists the bounded inverse operator $(I + \tilde{\mathcal R}(x))^{-1}$ on $B$ for each fixed $x \in [0,\pi]$.

\begin{lem} \label{lem:boundR}
For $\{ \la_u, v_u \}_J$ in $\mathcal S_{\Omega,\eps}$, the operator $(I + \tilde{\mathcal R}(x))^{-1}$ is uniformly bounded:
$$
\| (I + \tilde{\mathcal R}(x))^{-1} \|_{B \to B} \le C(\Omega, \eps), \quad x \in [0,\pi].
$$
\end{lem}

\begin{proof}
Let us prove the lemma by contradiction.
It can be shown that $\tilde{\mathcal R}(x)$ is continuous by $x \in[0,\pi]$ in the operator norm $\|.\|_{B \to B}$, so 
\begin{equation} \label{uniboundR}
\sup_{x \in [0,\pi]} \| (I + \tilde{\mathcal R}(x))^{-1} \|_{B \to B}< \infty
\end{equation}
for each fixed sequence $\{ \la_u, v_u \}_J$. Suppose that there exists a sequence of spectral data $\{ \la_u^{(p)}, v_u^{(p)} \}_J$, $p \ge 1$, such that
\begin{equation} \label{unboundRp}
\lim_{p \to \infty} \sup_{x \in [0,\pi]} \| (I + \tilde{\mathcal R}^{(p)}(x))^{-1} \|_{B \to B} = \infty.
\end{equation}

The following part of the proof is divided into three steps.

\smallskip

\underline{\textsc{Step 1:}} Let us construct the limit data $\{ \la_u, v_u \}_J$ of $\{ \la_u^{(p)}, v_u^{(p)} \}_J$ as $p \to \infty$ and obtain their asymptotics. 

Due to Definition~\ref{def:B} of $\mathcal S_{\Omega, \eps}$, the sequences $\{ \hat \rho_u^{(p)} \}_J$ ($\hat \rho_u^{(p)} = \rho_u^{(p)} - \tilde \rho_u$) and $\{ v_u^{(p)} \}_J$ are uniformly bounded with respect to $p \ge 1$:
$$
|\hat \rho_u^{(p)}| \le \Omega, \quad \| v_u^{(p)} \| \le C(\Omega), 
$$
for all $u \in J$. Therefore, we can extract weakly (i.e. for each fixed $u$) converging subsequences:
\begin{equation} \label{limp2}
\lim_{p \to \infty} \hat \rho_u^{(p)} =: \hat \rho_u, \quad \lim_{p \to \infty} v_u^{(p)} =: v_u.
\end{equation}

Put $\rho_u := \tilde \rho_u + \hat \rho_u$, $\la_u := \rho_u^2$. Let us show that $\{ \la_u, v_u \}_J$ are the spectral data for a problem \eqref{eqsi}--\eqref{bcsi} with some $\sigma \in L_2((0,\pi); \mathbb C^{m \times m})$ and $\check{H} \in \mathbb C^{m \times m}$. For this purpose, we have to prove that the conditions of Theorem~\ref{thm:charsi} are fulfilled.

Due to \eqref{asymptla}, \eqref{asymptV}, and \eqref{boundabove}, there hold
\begin{equation} \label{estrho}
|\rho_{nk}^{(p)} - (n-1)| \le \frac{\Omega}{n}, \quad \left\| (V_n^{(p)})^* V_n^{(p)} - \frac{2}{\pi} I_m\right\| \le \frac{\Omega}{n}.
\end{equation}
Hence, the same estimates are valid for $\{ \rho_{nk}, v_{nk} \}$, so
\begin{gather} \label{arhod}
\rho_{nk} = n - 1 + O\bigl(n^{-1}\bigr), \quad k = \overline{1,m}, \qquad 
V_n^* V_n = \frac{2}{\pi} I_m + O\bigl( n^{-1}\bigr), \quad n \ge 1.
\end{gather}

\smallskip

\underline{\textsc{Step 2.}} Let us prove that the sequence $\{ \chi_u \}_J$, $\chi_u := v_u \cos (\rho_u x)$, is complete in $L_2((0,\pi); \mathbb C^m)$.

Suppose that, on the contrary, there exists a non-zero vector function $w \in L_2((0,\pi); \mathbb C^m)$ that is orthogonal to all $\chi_u$, that is, 
\begin{equation} \label{ortu}
(w, \chi_u) = 0, \quad u \in J.
\end{equation}

For each $p \ge 1$, the sequence $\{\chi_u^{(p)} \}_J$ is a Riesz basis and \eqref{RBbound} holds. Hence, the vector function $w$ can be expanded as
$$
w = \sum_{u \in J} w_u^{(p)} \chi_u^{(p)}
$$
and
\begin{equation} \label{boundu}
\| \{ w_u^{(p)} \} \|_{l_2} \le \eps^{-1} \| w \|_{L_2}
\end{equation}
for each $p \ge 1$. Without loss of generality, assume that $\| w \|_{L_2} = 1$.

Using \eqref{ortu}, we obtain
\begin{align}
\nonumber 0 & = \left( w, \sum_{n,k} w_{nk}^{(p)} \chi_{nk} \right) = 
\| w \|^2_{L_2} + \left( w, \sum_{n,k \colon n \le N} w_{nk}^{(p)} (\chi_{nk} - \chi_{nk}^{(p)}) \right) + \left( w, \sum_{n,k \colon n > N} w_{nk}^{(p)} \mathring{\chi}_{nk} \right) \\ \label{sumser7} & + \left(w, \sum_{n,k \colon n > N} w_{nk}^{(p)} (\chi_{nk} - \mathring{\chi}_{nk}) \right) - \left( w, \sum_{n,k \colon n > N} w_{nk}^{(p)} \mathring{\chi}_{nk}^{(p)} \right) - \left( w, \sum_{n,k \colon n >N} w_{nk}^{(p)} (\chi_{nk}^{(p)} - \mathring{\chi}_{nk}^{(p)})\right),
\end{align}
where 
$$
\mathring{\chi}_{nk} = v_{nk} \cos((n-1)x), \quad \mathring{\chi}_{nk}^{(p)} = v_{nk}^{(p)} \cos((n-1)x).
$$

Using \eqref{arhod} and \eqref{boundu}, we get the estimate
\begin{equation} \label{est1}
\left\| \sum_{n,k \colon n > N} w_{nk}^{(p)} (\chi_{nk} - \mathring{\chi}_{nk})\right\|_{L_2} \le \frac{C}{N},
\end{equation}
which is uniform by $p \ge 1$.

For each $n \ge 1$, denote by $w_n^{(p)}$ the column vector of the elements $\{ w_{nk}^{(p)} \}_{k = 1}^m$ and put $\mathring{w}_n^{(p)} := V_n w_n^{(p)}$. In view of \eqref{arhod}, there holds 
\begin{equation} \label{normun}
\| \mathring{w}_n^{(p)} \| = \| w_n^{(p)} \| \Bigl(\sqrt{\tfrac{2}{\pi}} + O(n^{-1})\Bigr), \quad n > N,
\end{equation}
where the $O$-estimate is uniform by $p \ge 1$. Then
$$
\left( w, \sum_{n,k \colon n > N} w_{nk}^{(p)} \mathring{\chi}_{nk} \right) =  \left( w, \sum_{n > N} \mathring{w}_n^{(p)} \cos((n-1)x) \right) = \sum_{n > N} (w_n^0, \mathring{w}_n^{(p)}),
$$
where $\{ w_n^0 \}$ are vector Fourier coefficients of the vector-function $w$:
$$
w_n^0 := \int_0^{\pi} w(x) \cos((n-1)x) \, dx.
$$
Using the Cauchy-Bunyakosvky inequality, \eqref{boundu} and \eqref{normun}, we get
\begin{equation} \label{est2}
\sum_{n > N} (w_n^0, \mathring{w}_n^{(p)}) \le \| \{ w_{nk}^0 \}_{n > N} \|_{l_2} \| \{ \mathring{w}_{nk}^{(p)} \} \|_{l_2} \le C \| \{ w_{nk}^0 \}_{n > N} \|_{l_2}, \quad p \ge 1.
\end{equation}
Note that $\| \{ w_{nk}^0 \}_{n > N} \|_{l_2} \to 0$ as $N \to \infty$. Analogously to \eqref{est1} and \eqref{est2}, one can estimate the last two terms in \eqref{sumser7}.

Thus, for every $\eta > 0$, one can choose an index $N$ such that
\begin{equation} \label{est3}
\def\arraystretch{3}
\left.
\begin{array}{c}
\Biggl| \left( w, \sum\limits_{n,k \colon n > N} w_{nk}^{(p)} \mathring{\chi}_{nk} \right) \Biggr| \le \dfrac{\eta}{5}, \quad
\Biggl|\left(w, \sum\limits_{n,k \colon n > N} w_{nk}^{(p)} (\chi_{nk} - \mathring{\chi}_{nk}) \right)  \Biggr| \le \dfrac{\eta}{5}, \\
\Biggl| \left( w, \sum\limits_{n,k \colon n > N} w_{nk}^{(p)} \mathring{\chi}_{nk}^{(p)} \right) \Biggr| \le \dfrac{\eta}{5}, \quad
\Biggl| \left( w, \sum\limits_{n,k \colon n >N} w_{nk}^{(p)} (\chi_{nk}^{(p)} - \mathring{\chi}_{nk}^{(p)})\right) \Biggr| \le \dfrac{\eta}{5},
\end{array} \qquad \right\}
\end{equation}
for each $p \ge 1$. It follows from \eqref{limp2} that 
$\lim\limits_{p \to \infty} \chi_u^{(p)} = \chi_u$
for each fixed $u \in J$.
Hence, one can choose $p$ (for the fixed $N$) such that
\begin{equation} \label{est4}
\left\| \sum_{n,k \colon n \le N} w_{nk}^{(p)} (\chi_{nk} - \chi_{nk}^{(p)})\right\|_{L_2} \le \frac{\eta}{5}.
\end{equation}

Combining \eqref{sumser7}, \eqref{est3}, and \eqref{est4}, we arrive at the inequality
$$
0 = \left( w, \sum_{u \in J} w_u^{(p)} \chi_u \right) \ge 1 -\eta,
$$
where $\eta > 0$ can be arbitrarily small. This contradiction shows the completeness of $\{ \chi_u \}_J$.

\smallskip

\underline{\textsc{Step 3:}} We will obtain the operator $\tilde{\mathcal R}(x)$, study its properties, and show that $\tilde{\mathcal R}^{(p)}(x)$ tends to $\tilde{\mathcal R}(x)$ in the operator norm as $p\to \infty$.

Applying Theorem~\ref{thm:charsi}, we conclude that the limit values $\{ \la_u, v_u \}_J$ 
are the spectral data of some self-adjoint problem of form \eqref{eqsi}--\eqref{bcsi}.
Consider the operator $\tilde{\mathcal R}(x)$, which is constructed by using $\{ \la_u, \be_u \}_J$ ($\be_u := v_u v_u^*$). According to Theorem~\ref{thm:maineq} together with Remark~\ref{rem:maineq}, there exists a bounded inverse operator $(I + \tilde{\mathcal R}(x))^{-1}$ in $B$ for each fixed $x \in [0,\pi]$. According to the above arguments, it satisfies \eqref{uniboundR}.

Let us estimate the norm of the difference:
\begin{align} \nonumber
\| \tilde{\mathcal R}^{(p)}(x) - \tilde{\mathcal R}(x) \|_{B \to B} = & \sup_{n,k} \sum_{l,s} \| \tilde{\mathcal R}_{ls,nk}^{(p)}(x) - \tilde{\mathcal R}_{ls,nk}(x) \| \le \sup_{n,k} \Biggl( \sum_{l,s \colon l\le N} \| \tilde{\mathcal R}_{ls,nk}^{(p)}(x) - \tilde{\mathcal R}_{ls,nk}(x) \| \\ \label{normR} & + \sum_{l,s \colon l > N} \| \tilde{\mathcal R}_{ls,nk}^{(p)}(x) \| + \sum_{l,s \colon l > N} \| \tilde{\mathcal R}_{ls,nk}(x) \| \Biggr). 
\end{align}

It follows from \eqref{estpsiRt} that
\begin{equation} \label{estR}
\| \tilde{\mathcal R}^{(p)}_{ls,nk}(x) \| \le \frac{C(\Omega)}{l (|n-l| + 1)}.
\end{equation}
Consequently,
\begin{equation} \label{estsumR}
\sum_{l,s \colon l > N} \| \tilde{\mathcal R}_{ls,nk}^{(p)}(x) \| \le C \sqrt{\sum_{l = N + 1}^{\infty} \frac{1}{l^2}} \le \frac{C}{\sqrt{N}}
\end{equation}
for all $n \ge 1$, $k = \overline{1,m+1}$, $p \ge 1$.
Using the asymptotics \eqref{arhod}, we obtain the estimates analogous to \eqref{estR} and \eqref{estsumR} for $\tilde{\mathcal R}_{ls,nk}(x)$.

Let us proceed to the first sum in the right-hand side of \eqref{normR}. Let $s$ and $k$ be fixed integers in $\{ 1,  2, \dots, m \}$. The other cases can be studied similarly. Using \eqref{defR}, we get
\begin{align*}
\tilde{\mathcal R}_{ls,nk}^{(p)}(x) - \tilde{\mathcal R}_{ls,nk}(x) & = \hat \rho_{ls}^{(p)} \be_{ls}^{(p)} \tilde W_n(x, \rho_{ls}^{(p)}, \rho_{nk}^{(p)}) - \hat \rho_{ls} \be_{ls} \tilde W_n(x, \rho_{ls}, \rho_{nk}) \\ & = \mathscr S_1 + \mathscr S_2 + \mathscr S_3 + \mathscr S_4,
\end{align*}
where
\begin{align*}
\mathscr S_1 & := (\rho_{ls}^{(p)} - \rho_{ls}) \be_{ls}^{(p)} \tilde W_n(x, \rho_{ls}^{(p)}, \rho_{nk}^{(p)}), \\
\mathscr S_2 & := \rho_{ls} (\be_{ls}^{(p)} - \be_{ls}) \tilde W_n(x, \rho_{ls}^{(p)}, \rho_{nk}^{(p)}), \\
\mathscr S_3 & := \rho_{ls} \be_{ls} \bigl(\tilde W_n(x, \rho_{ls}^{(p)}, \rho_{nk}^{(p)}) - \tilde W_n(x, \rho_{ls}, \rho_{nk}^{(p)}) \bigr), \\
\mathscr S_4 & := \rho_{ls} \be_{ls} \bigl( \tilde W_n(x, \rho_{ls}, \rho_{nk}^{(p)}) - \tilde W_n(x, \rho_{ls}, \rho_{nk}) \bigr)
\end{align*}
Taking \eqref{defDt} and \eqref{defWnt} into account, we obtain the estimates
$$
\| \mathscr S_1 \|, \| \mathscr S_3 \| \le C |\rho_{ls}^{(p)} - \rho_{ls}|, \quad \| \mathscr S_2 \| \le C \|\be_{ls}^{(p)} - \be_{ls}\|, \quad \| \mathscr S_4 \| \le C |\rho_{nk}^{(p)} - \rho_{nk}|,
$$
which are uniform by $n,l,p \ge 1$ and $x \in [0,\pi]$.
Fix $\eta > 0$.
In view of \eqref{estrho} and \eqref{arhod}, one can choose $N$ such that, for each $n > N$, there holds $|\rho_{nk}^{(p)} - \rho_{nk}| \le \frac{\eta}{4}$. It follows from \eqref{limp2} that $\rho_{ls}^{(p)} \to \rho_{ls}$ and $\be_{ls}^{(p)} \to \be_{ls}$ as $p \to \infty$ for fixed $l,s$. Therefore, there exists $p \ge 1$ such that 
$$
|\rho_{ls}^{(p)} - \rho_{ls}| \le \frac{\eta}{4}, \quad \|\be_{ls}^{(p)} - \be_{ls}\| \le \frac{\eta}{4}, \quad l \le N, \quad s = \overline{1,m}.
$$
Then, we get
$$
\| \tilde{\mathcal R}_{ls,nk}^{(p)}(x) - \tilde{\mathcal R}_{ls,nk}(x) \| \le C \eta, \quad l \le N, \quad n \ge 1,
$$
where $\eta > 0$ can be chosen arbitrarily small. Summarizing the above arguments, we conclude that 
$$
\lim_{p \to \infty} \| \tilde{\mathcal R}^{(p)}(x) - \tilde{\mathcal R}(x) \|_{B \to B} = 0
$$
uniformly by $x \in [0,\pi]$.
Consequently,
$$
\lim_{p \to \infty} \| (I + \tilde{\mathcal R}^{(p)}(x))^{-1} \|_{B \to B} = \| (I + \tilde{\mathcal R}(x))^{-1} \|_{B \to B} < \infty, \quad x \in [0,\pi],
$$
which contradicts to \eqref{unboundRp} and so concludes the proof.
\end{proof}

\begin{lem} \label{lem:unibound}
Let spectral data $\{\la_u, v_u \}_J$ satisfy the conditions of Theorem~\ref{thm:char} and the estimates \eqref{boundabove} for the remainders of the asymptotics \eqref{asymptla} and \eqref{asymptV} with a constant $\Omega > 0$.
Furthermore, suppose that 
\begin{equation} \label{boundRK}
\| (I + \tilde{\mathcal R}(x))^{-1} \|_{B \to B} \le K, \quad x \in [0,\pi],
\end{equation}
for some $K > 0$. Then, the parameters $(Q, h, H)$ of the corresponding eigenvalue problem \eqref{eqv}--\eqref{bc} lie in $\mathcal P_R$, where $R > 0$ depends only on $\Omega$ and $K$.
\end{lem}

\begin{proof}
The proof repeats the arguments of \cite[Section 5]{Bond24} taking into account the specific construction of the operator $\tilde{\mathcal R}(x)$ for the matrix Sturm-Liouville operator in Section~\ref{sec:maineq}, so we outline it briefly. 

In view of Lemma~\ref{lem:asymptbe}, the estimates \eqref{boundabove} imply \eqref{estOm}.
So, we deduce the estimates
\begin{gather} \label{estpsit}
\| \tilde \psi_{nk}(x) \| \le C(\Omega), \quad 
\| \tilde \psi_{nk}'(x) \| \le C(\Omega) n, \\ \label{estRt}
\| \tilde {\mathcal R}_{ls,nk}(x) \| \le \frac{C(\Omega) \xi_l}{|n-l|+1}, \quad
\| \tilde{\mathcal R}'_{ls,nk}(x) \| \le C(\Omega) \xi_l,
\end{gather}
for $l,n \ge 1$, $k,s = \overline{1,m+1}$, $x \in [0,\pi]$, and
$\| \{ n\xi_n \} \|_{l_2} \le C(\Omega)$. 

Define the operator
$$
\mathcal R(x) := I - (I + \tilde{\mathcal R}(x))^{-1}.
$$
It follows from \eqref{boundRK} that $\| \mathcal R(x) \|_{B \to B} \le K + 1$ for each fixed $x \in [0,\pi]$. Using this estimate together with \eqref{opR} and \eqref{estRt}, we obtain
\begin{gather} \label{estR1}
\| \mathcal R_{ls,nk}(x) \| \le C \xi_l \left( \frac{1}{|n-l|+1} + \tau_n\right), \quad \| \mathcal R_{ls,nk}(x) \| \le C \xi_l \left( \frac{1}{|n-l|+1} + \tau_l \right), \\ \label{estR2}
\| \mathcal R'_{ls,nk}(x) \| \le C \xi_l,
\end{gather}
where $C = C(\Omega, K)$ and
$$
\tau_n := \sqrt{\sum_{s = 1}^{\infty} \frac{1}{s^2 (|n-s| + 1)^2}}, \quad n \ge 1.
$$

Next, construct the vector
$$
\psi(x) := \tilde \psi(x)(I - \mathcal R(x)).
$$
In the element-wise form, this yields
\begin{equation} \label{psink}
\psi_{nk}(x) = \tilde \psi_{nk}(x) - \sum_{l,s} \tilde \psi_{ls}(x) \mathcal R_{ls,nk}(x), \quad n \ge 1, \, k = \overline{1,m+1}, \, x \in [0,\pi].
\end{equation}
Using \eqref{psink} together with \eqref{estpsit}, \eqref{estR1}, and \eqref{estR2}, we deduce
\begin{equation} \label{estpsi1}
\def\arraystretch{1.7}
\left.
\begin{array}{l}
\| \psi_{nk}(x) \| \le C, \quad \| \psi_{nk}(x) - \tilde \psi_{nk}(x) \| \le C \tau_n, \\ 
\| \psi_{nk}'(x) \| \le C n, \quad \| \psi_{nk}'(x) - \tilde \psi_{nk}'(x) \| \le C,
\end{array} \quad \right\}
\end{equation}
where $C = C(\Omega, K)$, $n \ge 1$, $k = \overline{1,m+1}$, and $x \in [0,\pi]$.

Using the matrix functions $\{ \psi_{nk}(x) \}$ and the relations \eqref{defpsi}, we obtain the matrix functions $\vv_{nk}(x) := \vv(x, \la_{nk})$ for $k = \overline{1,m}$ and $\vv_{n,m+1}(x) := \vv(x, \tilde \la_n)$ by formulas \eqref{recvv}.
It follows from \eqref{estpsi1} that
\begin{equation} \label{estvv}
\def\arraystretch{1.7}
\left.
\begin{array}{l}
\| \vv_{nk}(x) \| \le C, \quad \| \vv_{nk}(x) - \tilde \vv_{nk}(x) \| \le C \tau_n, \\
\|\vv_{nk}'(x) \| \le C n, \quad \|\vv_{nk}'(x) - \tilde \vv_{nk}'(x) \| \le C,
\end{array} \qquad \right\}
\end{equation}
where $C = C(\Omega,K)$, $n \ge 1$, $k = \overline{1,m+1}$, $x \in [0,\pi]$. 

By virtue of Proposition~\ref{prop:relQhH}, the matrices $Q(x)$, $h$, and $H$ satisfy the relation \eqref{relQhH}, where the matrix function $E_0(x)$ in the case  $\tilde{\mathcal L} = \mathcal L(0_m,0_m,0_m)$ has the form
\begin{equation*} 
E_0(x) = \sum_{n = 1}^{\infty} \left( \sum_{k = 1}^m \vv_{nk}(x) \be_{nk} \tilde \vv_{nk}(x) - \vv_{n,m+1}(x) \tilde \be_n \tilde \vv_{n,m+1}(x) \right).
\end{equation*}

Using \eqref{estvv} together with the relation \eqref{defvvt} for $\tilde \vv(x, \la)$, the asymptotics \eqref{asymptla}, \eqref{asymptbe} and the estimates \eqref{estOm}, one can show that
$$
\| E_0 \|_{W_2^1} \le C(\Omega, K).
$$
Hence, the relations \eqref{relQhH} imply
$$
\| Q \|_{L_2} \le C(\Omega, K), \quad \| h \| \le C(\Omega, K), \quad
\| H \| \le C(\Omega, K),
$$
which concludes the proof.
\end{proof}

Obviously, Lemmas~\ref{lem:boundR} and \ref{lem:unibound} together imply Theorem~\ref{thm:boundinv}.

\section{Uniform stability for the inverse problem}\label{sec:unistab}

In this section, Theorem \ref{thm:unistab} is proved.
We estimate the difference between two coefficient triples $(Q, h, H)$ and $(\tilde Q, \tilde h, \tilde H)$ in $\mathcal P_R$ using the relations \eqref{relQhH} and a partition $\mathscr P$ from Definition~\ref{def:Js}. The obtained estimate together with the uniform boundedness of Inverse Problem~\ref{ip:1} yield the uniform stability. It is remarkable that the proof of the uniform stability on $\mathcal P_R$ relies only on the reconstruction formulas (see Proposition~\ref{prop:relQhH}) and does not use the main equation \eqref{main}. Furthermore, we consider some examples, which illustrate applications of Theorem~\ref{thm:unistab}. In particular, we show the convergence of the finite data approximation to the inverse problem solution.

\begin{lem} \label{lem:estE0}
If $(Q, h, H)$ and $(\tilde Q, \tilde h, \tilde H)$ belong to $\mathcal P_R$, then 
$\| E_0(x) \|_{W_2^1} \le C(R) Z$,
where $Z$ is defined in \eqref{defZ} for a partition $\mathscr P$ satisfying Definition~\ref{def:Js}.
\end{lem}

\begin{proof}
Suppose that $(Q, h, H)$, $(\tilde Q, \tilde h, \tilde H)$, and a partition $\mathscr P = \{ J_s \}_{s \ge 1}$ satisfy the hypotheses of the lemma.  
Using the notations from Definition~\ref{def:Js},
represent the matrix function \eqref{defE0} in the form
\begin{equation} \label{sumEs}
E_0(x) = \sum_{s = 1}^{\infty} E_s(x), \quad
E_s(x) := 
\sum_{u \in J_s} \bigl( \vv(x, \la_u) \be_u \tilde \vv^*(x, \la_u) - \vv(x, \tilde \la_u) \tilde \be_u \tilde \vv^*(x, \tilde \la_u) \bigr).
\end{equation}

Grouping the terms implies
\begin{multline} \label{defEs}
E_s(x) = \sum_{u \in J_s} \bigl(\vv(x, \la_u) - \vv(x, \la_{u_s})\bigr) \be_u \tilde \vv^*(x, \la_u) + \sum_{u \in J_s} \vv(x, \la_{u_s}) \be_u \bigl(\tilde \vv^*(x, \la_u) - \tilde \vv^*(x, \la_{u_s})\bigr) \\
- \sum_{u \in J_s} \bigl(\vv(x, \tilde \la_u) - \vv(x, \tilde \la_{u_s})\bigr) \tilde \be_u \tilde \vv^*(x, \tilde \la_u) - \sum_{u \in J_s} \vv(x, \tilde\la_{u_s}) \tilde \be_u \bigl(\tilde \vv^*(x, \tilde \la_u) - \tilde \vv^*(x, \tilde \la_{u_s})\bigr) \\
+ \bigl(\vv(x, \la_{u_s}) - \vv(x, \tilde \la_{u_s})\bigr) \be(J_s) \tilde \vv^*(x, \la_{u_s}) + \vv(x, \tilde \la_{u_s}) (\be(J_s) - \tilde \be(J_s)) \tilde \vv^*(x, \la_{u_s})  \\ + \vv(x, \tilde \la_{u_s}) \tilde \be(J_s) \bigl( \tilde \vv^*(x, \la_{u_s}) - \tilde \vv^*(x, \tilde \la_{u_s})\bigr).
\end{multline}

Using the asymptotics \eqref{asymptvv} and \eqref{asymptla}, the estimates \eqref{boundabove} and $\| Q \|_{L_2} + \| h \| \le R$, Schwarz's Lemma (see \cite[Lemma~1.6.1]{FY01}), and \eqref{defzeta}, we obtain
\begin{gather*}
\| \vv(x, \la_u) \| \le C, \quad
\|\vv(x, \la_u) - \vv(x, \la_w) \| \le C |\rho_u - \rho_w| \le C \zeta_s, 
\end{gather*}
where $\zeta_s := \zeta(J_s)$ \eqref{defzeta}, $u,w \in J_s$ and $x \in [0,\pi]$. Obviously, the similar estimates hold if $\la_u$ is replaced by $\tilde \la_u$ and/or $\vv$ is replaced by $\tilde \vv^*$.
Here and below in this proof, we mean that $C = C(R)$.

Using \eqref{asymptbe}, \eqref{defzeta}, and $\be_u \ge 0$, we get
\begin{equation} \label{bej}
\|\be_u \|, \| \tilde \be_u \| \le C, \: u \in J_s, \qquad \| \be(J_s) - \tilde \be(J_s) \| \le \zeta_s.
\end{equation}

Consequently, there holds
$$
\| E_s(x) \| \le C \zeta_s, \quad s \ge 1, \quad x  \in [0,\pi].
$$

Recalling $Z = \| \{ s \zeta_s \} \|_{l_2}$ \eqref{defZ} and \eqref{sumEs}, we obtain
\begin{equation} \label{estE0}
\| E_0(x) \| \le C \sum_{s = 1}^{\infty} \zeta_s \le C Z.
\end{equation}

Proceed to estimating the derivative 
\begin{equation} \label{serE0p}
E_0'(x) = \sum\limits_{s = 1}^{\infty} E_s'(x). 
\end{equation}
Consider the first term of the derivative of \eqref{defEs}:
\begin{equation} \label{Psj}
P_u(x) := \bigl( \vv'(x, \la_u) - \vv'(x, \la_{u_s})\bigr) \beta_u \tilde \vv^*(x, \la_u), \quad u \in J_s.
\end{equation}
Analogously to \eqref{asymptvv}, we have
\begin{equation} \label{asymptvvd}
\vv'(x, \rho^2) = -\rho \sin \rho x \, I_m + O (\exp(|\mbox{Im}\,\rho|x)),
\end{equation}
where the $O$-estimate is uniform on $\mathcal P_R$. Using \eqref{asymptvv}, \eqref{defzeta}, \eqref{bej}, \eqref{Psj}, and \eqref{asymptvvd}, we deduce
\begin{equation} \label{smPj}
P_u(x) = \bigl( \rho_{u_s} \sin \rho_{u_s} x - \rho_u \sin \rho_u x\bigr) \be_u \cos \rho_u x + r_u(x), \quad \| r_u(x) \| \le C \zeta_s, \quad x \in [0,\pi].
\end{equation}

Without loss of generality, we may assume that, for all sufficiently large $s$ ($s \ge s_*$), the inclusions $(n_1, k_1) \in J_s$ and $(n_2, k_2) \in J_s$ imply $n_1 = n_2 =: n(s)$. Otherwise, in view of \eqref{asymptla}, \eqref{defzeta}, and \eqref{defZ}, we have $Z = \infty$ and the assertion of the lemma holds automatically. For $s < s_*$, put $n(s) = 1$. Then \eqref{smPj} can be reduced to the form
$$
P_u(x) = k_u \be_u \sin (2 n(s) x) + z_u(x),
$$
where
$$
|k_u| \le C n(s) |\rho_u - \rho_{u_s}|, \quad \| z_u(x) \| \le C \zeta_s, \quad u \in J_s, \: x \in [0,\pi].
$$
Hence $\{ k_u \}_J \in l_2$ and $\|\{ k_u \}_J \|_{l_2} \le CZ$, so the trigonometric series $\sum\limits_{s = 1}^{\infty} \sum\limits_{u \in J_s} k_u \be_u \sin (2 n(s) x)$ converges in $L_2((0,\pi); \mathbb C^{m \times m})$ and its norm is bounded by $C Z$. The series $\sum\limits_{u \in J} z_u(x)$ of remainder terms converges absolutely and uniformly by $x \in [0,\pi]$ and has the same upper bound. Thus
$$
\left\| \sum_{u \in J} P_u(x) \right\|_{L_2} \le C Z.
$$
Analyzing analogously the other terms in \eqref{defEs}, we conclude that the series \eqref{serE0p} converges in $L_2((0,\pi); \mathbb C^{m \times m})$ and $\| E_0'(x) \|_{L_2} \le C Z$. Together with \eqref{estE0}, this completes the proof.
\end{proof}

Combining Theorem~\ref{thm:boundinv}, Lemma~\ref{lem:estE0}, and Proposition~\ref{prop:relQhH}, we arrive at Theorem~\ref{thm:unistab}. Proceed to examples illustrating the latter theorem. 

\begin{example} \label{ex:2}
This example shows the significance of using non-trivial partitions of eigenvalues into groups in accordance with Definition~\ref{def:Js}, not limiting by the partition from Example~\ref{ex:1}.
For $m = 2$, consider a family of boundary value problems $\mathcal L_{\de}$ and $\tilde{\mathcal L}_{\de}$, depending on a small parameter $\de > 0$ and having the corresponding spectral data:
\begin{gather*}
    \rho_{n1}^{\de} = \tilde \rho_{n1}^{\de} = n - 1 - \frac{\de}{n^2}, \quad \rho_{n2}^{\de} = \tilde \rho_{n2}^{\de} = n - 1 + \frac{\de}{n^2} \\
    v_{n1}^{\de} = \sqrt{\frac{2}{\pi}} \begin{bmatrix} 1 \\ 0 \end{bmatrix}, \quad v_{n2}^{\de} = \sqrt{\frac{2}{\pi}} \begin{bmatrix} 0 \\ 1 \end{bmatrix}, \quad
    \tilde v_{n1}^{\de} = \frac{1}{\sqrt{\pi}} \begin{bmatrix} 1 \\\ 1\end{bmatrix}, \quad \tilde v_{n2}^{\de} = \frac{1}{\sqrt{\pi}} \begin{bmatrix} 1 \\ -1 \end{bmatrix}.
\end{gather*}

Thus, the spectra of $\mathcal L_{\de}$ and $\tilde{\mathcal L}_{\de}$ coincide with each other, and they are obtained by
splitting the eigenvalues of the problem $\mathcal L_0 = \mathcal L(0_m,0_m,0_m)$. The full-rank weight matrices split according to \eqref{findal} into different orthogonal bases $\{ v_{n1}^{\de}, v_{n2}^{\de} \}$ and $\{ \tilde v_{n1}^{\de}, \tilde v_{n2}^{\de} \}$ in $\mathbb C^2$.
Note that
$$
\be_n^{\de} = \be_{n1}^{\de} + \be_{n2}^{\de} = \be_n^0, \quad \be_{nk}^{\de} := v_{nk}^{\de} (v_{nk}^{\de})^*, 
$$
and similar relations are valid for $\tilde \be_{nk}^{\de}$.

In this case, the indices $J$ should not be divided into one-element groups as in Example~\ref{ex:1}, since the series \eqref{Zex1} diverges. So, we compose two-element groups $J_n := \{ (n,k) \colon k = 1, 2 \}$, $n \ge 1$. Then 
\begin{align*}
\zeta(J_n) & = |\rho_{n1}^{\de} - \tilde \rho_{n1}^{\de}| + 
|\rho_{n2}^{\de} - \tilde \rho_{n2}^{\de}| + |\rho_{n1}^{\de} - \rho_{n2}^{\de}| + |\tilde \rho_{n1}^{\de} - \tilde \rho_{n2}^{\de}| + \| \be_n - \tilde \be_n \| = \frac{4 \de}{n^2}, \\
Z & = 4 \de \sqrt{\sum_{n = 1}^{\infty} \frac{1}{n^2}}.
\end{align*}

Thus, Theorem~\ref{thm:unistab} implies the estimate
$$
\| Q_{\de} - \tilde Q_{\de} \|_{L_2} + \| h_{\de} - \tilde h_{\de} \| + \| H_{\de} - \tilde H_{\de} \| \le C \de, 
$$
whose left-hand side tends to zero as $\de \to 0$.

A feature of this example is that the both problems are obtained by splitting the eigenvalues of a ``virtual'' problem (in this case, of $\mathcal L_0$). Therefore, although all the eigenvalues of $\mathcal L_{\de}$ and $\tilde{\mathcal L}_{\de}$ are simple, we have to join them into groups in order to show the proximity of the potentials $Q_{\de}$ and $\tilde Q_{\de}$.
\end{example}

\begin{example}
Let $(Q, h, H) \in \mathcal P$ be fixed and let finite spectral data $\{ \la_{nk}, \be_{nk} \}_{n = \overline{1,p}, k = \overline{1,m}}$ ($p \in \mathbb N$) be given. Then, we can complete the spectral data as follows:
$$
\la_{nk}^{(p)} = \begin{cases}
                    \la_{nk}, & n \le p, \\
                    \tilde \la_{nk}, & n > p,
                \end{cases} \quad
\be_{nk}^{(p)} = \begin{cases}
                    \be_{nk}, & n \le p, \\
                    \tilde \be_{nk}, & n > p,
                \end{cases}
$$
where $\{ \tilde \la_{nk}, \tilde \be_{nk} \}_{n \ge 1, \, k = \overline{1,m}}$ are the spectral data of the problem $\tilde {\mathcal L} = \mathcal L(0_m,0_m,0_m)$, which is considered in Section~\ref{sec:maineq} in more detail.
Denote by $(Q_p, h_p, H_p)$ the so-called finite data approximation that is reconstructed by using $\mathscr S_p := \{ \la_u^{(p)}, \be_u^{(p)} \}_J$.
For each fixed $p \ge 1$, consider two spectral data sets $\mathscr S = \{ \la_u, \be_u \}_J$ and $\mathscr S_p$ together with the following partition $\mathscr P_p = \{ J_s^{(p)} \}_{s \ge 1}$ of the set $J$: 
\begin{align*}
    & J_{(n-1)m+k}^{(p)} := \{ (n,k) \},  \quad n = \overline{1,p}, \quad k = \overline{1,m}, \\
    & J_{p(m-1) + n}^{(p)} := \{ (n,k) \colon k = \overline{1,m} \}, \quad n > p.
\end{align*}
According to \eqref{defzeta}, we have
\begin{align*}
& \zeta(J_s^{(p)}) = 0, \quad s = \overline{1, mp}, \\
& \zeta(J_{p(m-1) + n}^{(p)}) = \sum_{k = 1}^m |\rho_{nk} - \tilde \rho_n| + \sum_{k = 2}^m |\rho_{nk} - \rho_{n1}| + \| \be_n - \tilde \be_n \|, \quad n > p.
\end{align*}
In view of the asymptotics \eqref{asymptla} and \eqref{asymptbe}, we get
$$
Z(\mathscr P_p, \mathscr S, \mathscr S_p) \le C \bigl( \| \{ \varkappa_{nk} \}_{n > p} \|_{l_2} + \|\{\mathscr K_n \}_{n > p} \|_{l_2}\bigr).
$$
The right-hand side obviously tends to zero as $p \to \infty$. Therefore, Theorem~\ref{thm:unistab} implies the convergence of finite data approximations to the parameters of the problem $\mathcal L$:
$$
\lim_{p \to \infty} \bigl( \| Q - Q_p \|_{L_2} + \| h - h_p \| + \| H - H_p \| \bigr) = 0.
$$
\end{example}

\section{General case} \label{sec:gen}

In this section, we remove the restriction \eqref{omega0}, and generalize Theorems~\ref{thm:bounddir}, \ref{thm:boundinv}, and \ref{thm:unistab} to this case.

Consider the problem $\mathcal L(Q, h, H)$ of form \eqref{eqv}--\eqref{bc} under the condition \eqref{sa}. Denote
\begin{equation} \label{omega}
\om := h + H + \frac{1}{2}\int_0^{\pi} Q(x) dx
\end{equation}
Without loss of generality, we may assume that
\begin{equation} \label{condom}
\om = \mbox{diag} \{ \om_k \}_{k = 1}^m, \quad \om_1 \le \om_2\le \dots \le\om_m.
\end{equation}
Indeed, it follows from \eqref{sa} and \eqref{omega} that $\om = \om^*$, so the condition \eqref{condom} can be achieved by a unitary transform.

In this section, we assume that $\omega$ is a fixed constant matrix satisfying \eqref{condom}. Denote by $\mathcal P_{\om, R}$ the set of triples $(Q, h, H)$ of class \eqref{classP} satisfying \eqref{sa}, \eqref{omega} and $\| h \| + \| H \| + \| Q \|_{L_2} \le R$.

\begin{thm} \label{thm:bounddir1}
Let $R > 0$ be fixed. Then the spectral data $\{ \la_u, v_u \}_J$ of the problem \eqref{eqv}--\eqref{bc} with parameters $(Q, h, H) \in \mathcal P_{\om, R}$ fulfill the asymptotic formulas \eqref{asymptV},
\begin{align} \label{asymptla1}
& \rho_{nk} = \sqrt{\la_{nk}} = n - 1 + \frac{\om_k}{\pi n} + \frac{\varkappa_{nk}}{n}, \quad \{ \varkappa_{nk} \} \in l_2, \\ \label{asymptV1}
& \sum_{s \colon \om_s = \om_k} v_{ns} v_{ns}^* = \frac{2}{\pi} I_{\langle k \rangle} + \mathscr K_{n\langle k \rangle}, \quad \{ \mathscr K_{n \langle k \rangle} \} \in l_2(\mathbb C^{m \times m}),
\end{align}
where $n \ge 1$, $k = \overline{1,m}$, and
$$
I_{\langle k \rangle} := [I_{\langle k \rangle ij}]_{i,j = 1}^m, \quad I_{\langle k \rangle ij} = \begin{cases}
                                    1, & i = j \: \text{and} \: \om_i = \om_k, \\
                                    0, & \text{otherwise}.
                                \end{cases}
$$
Furthermore, the sequence $\{ v_u \cos (\rho_u x) \}_J$ is a Riesz basis in $L_2((0,\pi); \mathbb C^m)$ and the estimates \eqref{boundabove}, \eqref{RBbound} hold together with
$$
\| \{ \mathscr K_{n\langle k \rangle} \}_{n \ge 1} \|_{l_2} \le \Omega, \quad k = \overline{1,m},
$$
where constants $\Omega > 0$ and $\eps > 0$ depend only on $R$
and $\omega$.
\end{thm}

Thus, the eigenvalues in the general case \eqref{condom} split into groups $\{ \la_{nk} \}_{k = 1}^m$ according to the first term of the asymptotics \eqref{asymptla1}, and these groups split into smaller groups $\{ \la_{ns} \colon \om_s = \om_k \}$ according to the second term. This makes the analysis of the inverse problem more technically complicated. 

The necessary and sufficient conditions for the solvability of Inverse Problem~\ref{ip:1} in the case \eqref{omega} are obtained from Theorem~\ref{thm:char} by replacing \eqref{asymptla} with \eqref{asymptla1} and adding the requirement \eqref{asymptV1} (see \cite{Bond19}). Indeed, in view of Definition~\ref{def:E} and \eqref{findal}, the condition \eqref{asymptV1} is equivalent to the asymptotics of \cite[Lemma 2]{Bond19} for the weight matrices:
\begin{equation} \label{asymptbens}
\sum_{s \colon \om_s = \om_k} \be_{ns} = \sum_{s \colon \om_s = \om_k} \al'_{ns} = \frac{2}{\pi} I_{\langle k \rangle} + \mathscr K_{n \langle k \rangle}, \quad n \ge 1, \: k = \overline{1,m}.
\end{equation}

For $\omega$ satisfying \eqref{condom}, $\Omega > 0$, and $\eps > 0$, denote by $\mathcal S_{\om,\Omega,\eps}$ the set of all the sequences $\{ \la_u, v_u \}_J$ of $\mathcal S$ satisfying all the assertions of Theorem~\ref{thm:bounddir1} with the initially fixed $\omega$, $\Omega$, and $\eps$.

\begin{thm} \label{thm:boundinv1}
Let $\Omega > 0$ and $\eps > 0$ be fixed. Then a sequence $\{ \la_u, v_u \}_J$ of $\mathcal S_{\om, \Omega, \eps}$ is the spectral data of a unique boundary value problem $\mathcal L(Q, h, H)$ with $(Q, h, H) \in \mathcal P_{\om,R}$, where $R$ depends only on $\om$, $\Omega$, and $\eps$. 
\end{thm}

\begin{defin}[Partition $\mathscr P'$] \label{def:Jsi}
Let $\om$ satisfy \eqref{condom} and let $\mathscr P = \{ J_s \}_{s \ge 1}$ be a partition introduced according to Definition~\ref{def:Js}. Define a partition $\mathscr P' = \{ J_{si} \}_{s \ge 1, \, i = \overline{1,m_s}}$ ($m_s \in \mathbb N$) of $J$ satisfying
$$
J_s = \bigcup_{i = 1}^{m_s} J_{si}, \quad J_{si} \cap J_{sj} = \varnothing, \quad i \ne j, \quad s \ge 1.
$$

Let $\{ \rho_u, \be_u \}_J$ and $\{ \tilde \rho_u, \tilde \be_u \}_J$ be $(\rho,\be)$-spectral data, induced by spectral data $\mathscr S$ and $\tilde{\mathscr S}$.
Denote
$$
\be(J_{si}) := \sum_{u \in J_{si}} \be_u, \quad \tilde \be(J_{si}) := \sum_{u \in J_{si}} \tilde{\be}_u,
$$
\vspace*{-20pt}
\begin{multline*}
\theta_s := \sum_{u \in J_s} |\rho_u - \tilde \rho_u| + 
\sum_{i = 1}^{m_s} \sum_{u, v \in J_{si}} \bigl( |\rho_u - \rho_v| + |\tilde \rho_u - \tilde \rho_v|\bigr) + \\ +
\frac{1}{s} \sum_{i = 1}^{m_s} \| \be(J_{si}) - \tilde \be(J_{si}) \| \biggr) + \| \be(J_s) - \tilde \be(J_s) \|,
\end{multline*}
where $\be(J_s)$ and $\tilde \be(J_s)$ are defined in \eqref{sumbe},
and put 
\begin{equation} \label{defTheta}
\Theta(\mathscr P', \mathscr S, \tilde{\mathscr S}) := \| \{ s \theta_s \}_{s \ge 1} \|_{l_2}.
\end{equation}
\end{defin}

The uniform stability of Inverse Problem~\ref{ip:1} in the general case \eqref{omega} is formulated as follows.

\begin{thm} \label{thm:unistab1}
Let $\mathcal L = \mathcal L(Q, h, H)$ and $\tilde{\mathcal L} = \mathcal L(\tilde Q, \tilde h, \tilde H)$ be two problems of form \eqref{eqv}--\eqref{bc}, whose spectral data $\mathscr S = \{ \la_u, v_u \}_J$ and $\tilde{\mathscr S} = \{ \tilde \la_u, \tilde v_u \}_J$, respectively, lie in $\mathcal S_{\om,\Omega,\eps}$. Then
$$
\| Q - \tilde Q \|_{L_2} + \|h - \tilde h\| + \|H - \tilde H \| \le C(\Omega, \eps) \Theta(\mathscr P', \mathscr S, \tilde{\mathscr S}),
$$
where $\mathscr P'$ is any partition satisfying Definition~\ref{def:Jsi}.
\end{thm}

Let us show that, for any two problems $\mathcal L$ and $\tilde{\mathcal L}$ with $\omega = \tilde \omega$, the partition $\mathscr P'$ can be chosen so that $\Theta < \infty$. Recall that \eqref{asymptV} is equivalent to \eqref{asymptbe} and \eqref{asymptV1}, to \eqref{asymptbens}. Denote by $\{ \varsigma_i \}_{i = 1}^q$ all the distinct numbers among $\{ \om_k \}_{k = 1}^m$ and put $J_{si} := \{ (s,k) \colon \om_k = \varsigma_i \}$, $s \ge 1$, $i = \overline{1,q}$. Then
$$
\be(J_{si}) = \sum_{k \colon \om_k = \varsigma_i} \be_{sk}, \quad
\be(J_s) = \be_s,
$$
so the asymptotics \eqref{asymptla1}, \eqref{asymptbe}, and \eqref{asymptbens} imply $\{ s \theta_s \} \in l_2$, which yields the claim.

The proof strategy for Theorems~\ref{thm:bounddir1}--\ref{thm:unistab1} is analogous to the one for Theorems~\ref{thm:bounddir}, \ref{thm:boundinv}, and \ref{thm:unistab}.
However, in the general case \eqref{omega}, the construction of the main equation is more complicated than the one provided in Section~\ref{sec:maineq} (see \cite{Bond19}), so we omit the proof details here.

\section{Sturm-Liouville operator on a graph} \label{sec:graph}

In this section, we apply our technique to the inverse Sturm-Liouville problem on the star-shaped graph that was studied in the papers \cite{Bond20-ipse, Bond20-amp}. We use the same notations, as in Sections~\ref{sec:prelim}--\ref{sec:unistab} for objects related to the matrix Sturm-Liouville problem \eqref{eqv}--\eqref{bc}, to denote analogous objects related to the problem on the graph in order to stress out the similarity of these two problems.

Consider a star-shaped graph consisting of $m$ edges of equal length $\pi$. Let each edge be equipped with a parameter $x_j \in (0,\pi)$ such that the value $x_j = 0$ corresponds to the boundary vertex and $x_j = \pi$, to the interior vertex.

Consider the eigenvalue problem for the Sturm-Liouville equations
\begin{equation} \label{eqvg}
-y_j''(x_j) + q_j(x_j) y_j(x_j) = \la y_j(x_j), \quad x_j \in (0,\pi), \quad j = \overline{1,m},
\end{equation}
on the edges of the star-shaped graph with the boundary conditions
\begin{equation} \label{bcg}
y_j(0) = 0, \quad j = \overline{1,m},
\end{equation}
and with the standard matching conditions at the interior vertex:
\begin{equation} \label{mc}
y_1(\pi) = y_j(\pi), \quad j =  \overline{2,m}, \qquad \sum_{j = 1}^m y_j'(\pi) = 0,
\end{equation}
where $y_j \in W_1^2[0,\pi]$, 
$\la$ is the spectral parameter, and $q_j$ ($j = \overline{1,m}$) are real-valued functions of $L_2(0,\pi)$ satisfying the condition
\begin{equation} \label{zeroqj}
\int_0^{\pi} q_j(x_j) \, d x_j = 0, \quad j = \overline{1,m}.
\end{equation}

The standard matching conditions \eqref{mc} are the special case of the so-called $\delta$-coupling (see \cite[Appendix~1]{Kur24}).
They express Kirchoff's law in electrical circuits, balance of tension in elastic string network, etc. (see, e.g., \cite{BCFK06, PPP04}). The condition \eqref{zeroqj} is imposed for technical simplicity, it can be removed as in Section~\ref{sec:gen}.

The boundary value problem \eqref{eqvg}--\eqref{mc} can be equivalently represented in the matrix form $\mathscr L = \mathscr L(Q)$:
\begin{gather} \label{eqv2}
    -Y''(x) + Q(x) Y(x) = \la Y(x), \quad x \in (0, \pi), \\ \label{bc2}
    Y(0) = 0, \quad T Y'(\pi) - T^{\perp} Y(\pi) =  0,
\end{gather}
where 
$$
Q(x) = \mbox{diag}\,\{ q_j(x) \}_{j = 1}^m, \quad
T = [T_{jk}]_{j,k = 1}^m, \quad T_{jk} = \frac{1}{m}, \quad T^{\perp} = I_m - T.
$$
Clearly, $T$ is an orthogonal projector of rank one and $T^{\perp}$ is the complementary projector. In view of \eqref{zeroqj}, we have
\begin{equation} \label{zeroQ}
Q \in L_2((0,\pi); \mathbb C^{m \times m}), \quad Q(x) = Q^*(x) \:\: \text{a.e. on} \:\: (0, \pi), \quad \int_0^{\pi} Q(x) \, dx = 0.
\end{equation}

Denote by $\mathcal P$ the class of the matrix functions satisfying \eqref{zeroQ} and by $\mathcal P^d$ the subclass of the diagonal matrix functions of $\mathcal P$. Below, we consider the problem $\mathscr L(Q)$ with arbitrary $Q \in \mathcal P$, not necessarily being diagonal, unless it is stated otherwise.

For any $Q \in \mathcal P$, the boundary value problem $\mathscr L(Q)$ is self-adjoint. Furthermore, it has the countable set of real eigenvalues $\{ \la_{nk} \}_{n \ge 1, \, k = \overline{1,m}}$ possessing the following asymptotics (see \cite[Theorem~7.4.7]{MP15}):
\begin{equation} \label{asymptla2}
\def\arraystretch{2}
\left.
\begin{array}{l}
\rho_{n1} := \sqrt{\la_{n1}} = n - \dfrac{1}{2} + \dfrac{\varkappa_{nk}}{n}, \\
\rho_{nk} := \sqrt{\la_{nk}} = n + \dfrac{\varkappa_{nk}}{n}, \quad k = \overline{2,m},
\end{array} \quad \right\} \quad n \ge 1.
\end{equation}

We assume that the eigenvalues are numbered counting with their multiplicities in the non-decreasing order: $\la_u \le \la_v$ if $u < v$, $u, v \in J$.

Denote by $\Phi(x, \la)$ the matrix solution of equation \eqref{eqv2} with the boundary conditions
$$
\Phi(0,\la) = I_m, \quad T \Phi'(\pi,\la) - T^{\perp} \Phi(\pi, \la) = 0_m.
$$
Put $M(\la) := \Phi'(0,\la)$ and introduce the weight matrices
\begin{equation} \label{defal2}
\al_u := -\Res_{\la = \la_u} M(\la), \quad u \in J.
\end{equation}
The sign minus is chosen so that $\al_u = \al_u^* \ge 0$.

\begin{prop}[\hspace*{-3pt}\cite{CM13, Xu19}] \label{prop:uniqg}
The data $\{\la_u,\al_u \}_J$ uniquely specify the potential $Q$ of the problem \eqref{eqv2}--\eqref{bc2}.
\end{prop}

It is worth mentioning that, for $Q \in \mathcal P^d$, Proposition~\ref{prop:uniqg} follows from the uniqueness theorems of \cite{BW05, Yur05} for inverse spectral problems on graphs.

Introduce the matrices $\{ \al'_u \}_J$ and the vectors $\{ \mathscr E_u \}_J$ by using $\{ \al_u \}_J$ \eqref{defal2} according to Definition~\ref{def:E}.
The spectral data characterization for the problem \eqref{eqv2}--\eqref{bc2} is formulated as follows:

\begin{prop}[\hspace*{-3pt}\cite{Bond20-amp}] \label{prop:charg}
For a sequence $\{ \la_u, \al_u \}_J$ to be the eigenvalues and the weight matrices of a problem $\mathscr L(Q)$ with $Q \in \mathcal P$, the following conditions are necessary and sufficient:
\begin{enumerate}
\item $\la_u \in \mathbb R$, $\la_u \le \la_w$ for $u < w$, and the relations \eqref{structal}.
\item Asymptotics \eqref{asymptla2} and
\begin{equation} \label{asymptbe2}
\al_{n1} = \frac{2 (n - \frac{1}{2})^2}{\pi} \left( T + \frac{\mathscr K^I_n}{n}\right), \quad \be_n := \sum_{k = 2}^m \al_{nk}' = \frac{2 n^2}{\pi} \left( T^{\perp} + \frac{\mathscr K^{II}_n}{n}\right), \quad n \ge 1,
\end{equation}
where $\{ \mathscr K_n^I \}, \, \{ \mathscr K_n^{II} \} \in l_2(\mathbb C^{m \times m})$.
\item The sequence $\{ \rho_u^{-1} \sin (\rho_u x) \mathscr E_u \}_J$ is complete in $L_2((0,\pi); \mathbb C^m)$.
\end{enumerate}
\end{prop}

Thus, the form of the sequence in condition 3 of Propositions~\ref{prop:charw} and~\ref{prop:charg}
depends on the boundary condition at $x = 0$. In the graph case $Q \in \mathcal P^d$, an additional condition is required for the diagonality of the matrix potential (see \cite[Theorem~3.4]{Bond20-amp}).

Let $\vv(x, \la)$ be the solution of equation \eqref{eqv2} satisfying the initial conditions 
\begin{equation} \label{icsin}
\vv(0, \la) = 0_m, \quad \vv'(0,\la) = I_m. 
\end{equation}

The boundary value problem $\mathscr L(Q)$ has an orthonormal basis of vector eigenfunctions $\{ Y_u \}_J$ in $L_2((0,\pi); \mathbb C^m)$. Obviously, they can be represented as
\begin{equation} \label{defv}
Y_{n1}(x) = \bigl(n - \tfrac{1}{2}\bigr)\vv(x, \la_{n1}) v_{n1}, \quad Y_{nk}(x) = n \vv(x, \la_{nk}) v_{nk}, \quad k \ge 2,
\end{equation}
where $v_u \in \mathbb C^m$, $u \in J$. In other words,
$$
v_{n1} := \bigl(n - \tfrac{1}{2}\bigr)^{-1} Y_{n1}'(0), \quad
v_{nk} := n^{-1} Y_{nk}'(0), \quad k \ge 2.
$$

The normalizing coefficients $\bigl(n - \tfrac{1}{2}\bigr)$ and $n$ in \eqref{defv} are chosen according to the main parts of the eigenvalue asymptotics \eqref{asymptla2} because
$$
\vv(x, \rho^2) = \frac{\sin \rho x}{\rho} I_m + O\left( \rho^{-2} \exp(|\mbox{Im}\, \rho| x)\right), \quad |\rho| \to \infty.
$$
Consequently, the vectors $v_u$ satisfy the bounds \eqref{boundv}, which are convenient from the technical viewpoint.

Consider the following inverse spectral problem.

\begin{ip} \label{ip:g}
Given $\{ \la_u,  v_u \}_J$, find $Q$.
\end{ip}

In view of \eqref{defv}, the norming vectors are related to the weight matrices as follows:
\begin{equation} \label{findal2}
\al_u = \sum_{w\colon \la_w = \la_u} \be_w, \quad u \in J,
\end{equation}
where
\begin{equation} \label{defvnk}
\be_{n1} := \left( n - \tfrac{1}{2}\right)^2 v_{n1} v_{n1}^*, \quad
\be_{nk} := n^2 v_{nk} v_{nk}^*, \quad k = \overline{2,m}.
\end{equation}

Consequently, Propositions~\ref{prop:uniqg} and~\ref{prop:charg} immediately yield the following corollaries on the uniqueness for solution of Inverse Problem~\ref{ip:g} and on the spectral data characterization, respectively.

\begin{cor}
The spectral data $\{ \la_u, v_u \}_J$ uniquely specify the potential $Q$ of the problem \eqref{eqv2}--\eqref{bc2}.
\end{cor}

Denote
\begin{equation} \label{defchi}
\chi_u := v_u \sin(\rho_u t) \:\: \text{if} \:\: \rho_u \ne 0 \:\: \text{and} \:\: \chi_u := v_u t \:\: \text{if} \:\: \rho_u = 0.
\end{equation}

\begin{cor}
For a sequence $\{ \la_u, v_u \}_J \in \mathcal S$ to be the spectral data of the problem \eqref{eqv2}--\eqref{bc2} with $Q \in \mathcal P$, the following conditions are necessary and sufficient:
\begin{enumerate}
\item Asymptotics \eqref{asymptla2}, \eqref{asymptV}, and 
\begin{equation} \label{asymptvn1} 
v_{n1} v_{n1}^* = \frac{2}{\pi} T + \frac{K_n^I}{n}, \quad \{ K_n^I\} \in l_2(\mathbb C^{m \times m})
\end{equation}
\item The sequence $\{ \chi_u \}_J$ defined by \eqref{defchi} is complete in $L_2((0,\pi); \mathbb C^m)$.
\end{enumerate}
\end{cor}

\begin{proof}
Let us show that the asymptotics \eqref{asymptV} and \eqref{asymptvn1} together are equivalent to \eqref{asymptbe2}. In view of Definition~\ref{def:E}, \eqref{findal2} and \eqref{defvnk}, the asymptotics \eqref{asymptbe2} can be rewritten as \eqref{asymptvn1} and
\begin{equation} \label{asymptvn2}
\sum_{k = 2}^m v_{nk} v_{nk}^* = \frac{2}{\pi} T^{\perp} + \frac{K_n^{II}}{n}, \quad \{ K_n^{II}\} \in l_2(\mathbb C^{m \times m}), \quad n \ge 1.
\end{equation}
Summing up \eqref{asymptvn1} and \eqref{asymptvn2}, we arrive at the formula
$$
V_n V_n^* = \frac{2}{\pi} I_m + \frac{K_n^{III}}{n}, \quad \{ K_n^{III} \} \in l_2(\mathbb C^{m \times m}), \quad n \ge 1,
$$
which yields \eqref{asymptV}. The other proof details are obvious.
\end{proof}

Proceed to the analogs of Theorems~\ref{thm:bounddir}, \ref{thm:boundinv}, and \ref{thm:unistab} for the problem \eqref{eqv2}--\eqref{bc2}. For $R > 0$, denote by $\mathcal P_R$ and $\mathcal P_R^d$ the sets of matrix functions $Q$ of $\mathcal P$ and $\mathcal P^d$, respectively, additionally satisfying $\| Q \|_{L_2} \le R$.

\begin{thm} \label{thm:bounddir2}
Let $R > 0$ be fixed. Then the spectral data $\{ \la_u, v_u \}_J$ of the problem \eqref{eqv2}--\eqref{bc2} with $Q \in \mathcal P_R$ fulfill the asymptotics \eqref{asymptV}, \eqref{asymptla2}, and \eqref{asymptvn1}
with the remainder terms satisfying the estimates \eqref{boundabove} and $\| \{ K_n^I \}_{n \ge 1} \|_{l_2} \le \Omega$, where $\Omega > 0$ depends only on $R$. Moreover, the sequence $\{\chi_u\}_J$ given by \eqref{defchi} is a Riesz basis in $L_2((0,\pi); \mathbb C^m)$ and the estimate \eqref{RBbound} holds with $\eps > 0$ depending only on $R$.
\end{thm}

Clearly, Theorem~\ref{thm:bounddir2} holds for any $Q \in \mathcal P_R^d$, that is, for the problem \eqref{eqvg}--\eqref{mc} on the star-shaped graph.

Analogously to Definition~\ref{def:B}, introduce the set $\mathcal S_{\Omega, \eps}$ of all the sequences $\{ \la_u, v_u\}_J$ in $\mathcal S$ satisfying all the conclusions of Theorem~\ref{thm:bounddir2} for fixed $\Omega$ and $\eps$.

\begin{thm} \label{thm:boundinv2}
Let $\Omega > 0$ and $\eps > 0$ be fixed. Then, any sequence $\{ \la_u, v_u \}_J$ of $\mathcal S_{\Omega, \eps}$ is the spectral data of the boundary value problem \eqref{eqv2}--\eqref{bc2} for a unique $Q \in \mathcal P$. Moreover, $Q \in \mathcal P_R$, where $R > 0$ depends only on $\Omega$ and $\eps$.
\end{thm}
 
Let $\mathscr P = \{ J_s \}_{s \ge 1}$ be a partition of the set \eqref{defJ} satisfying Definition~\ref{def:Js}. For $j \in \{ 1, 2, \dots, m \}$, denote by $\be_{jj}(J_s)$ and $\tilde \be_{jj}(J_s)$ the diagonal elements at the position $(j,j)$ of the matrices $\be(J_s)$ and $\tilde \be(J_s)$ defined by \eqref{sumbe}, respectively. 
Analogously to $\zeta(J_s)$ and $Z$ given by \eqref{defzeta} and \eqref{defZ}, respectively, introduce the following notations for $j = \overline{1,m}$: 
\begin{gather} \nonumber
\zeta_j(J_s) := \sum_{u \in J_s} |\rho_u - \tilde \rho_u| +  \sum_{u,v \in J_s} \bigl( |\rho_u - \rho_v| + |\tilde \rho_u - \tilde \rho_v| \bigr) + |\be_{jj}(J_s) - \tilde \be_{jj}(J_s)|, \\ \label{defZj}
Z_j(\mathscr P, \mathscr S, \tilde{\mathscr S}) := \| \{ s \zeta_j(J_s) \}_{s \ge 1} \|_{l_2}.
\end{gather}

Let us formulate the uniform stability theorem for the Sturm-Liouville inverse problem on the star-shaped graph.

\begin{thm} \label{thm:unistab2}
Let the spectral data $\mathscr S = \{ \la_u, v_u \}_J$ and $\tilde{\mathscr S} = \{ \tilde \la_u, \tilde v_u \}_J$ of two problems of form \eqref{eqvg}--\eqref{mc} with potentials $\{ q_j \}_{j = 1}^m$ and $\{ \tilde q_j \}_{j = 1}^m$ of $L_2(0,\pi)$, satisfying \eqref{zeroqj}, lie in $\mathcal S_{\Omega, \eps}$. Then
\begin{equation} \label{unistab2}
\| q_j - \tilde q_j \|_{L_2} \le C(\Omega, \eps) Z_j(\mathscr P, \mathscr S, \tilde{\mathscr S}), \quad j = \overline{1,m},
\end{equation}
for any partition $\mathscr P$ satisfying Definition~\ref{def:Js}.
\end{thm}

It is worth noting that the values $Z_j$ depend only on the diagonal elements of the matrices $\be_u$ and $\tilde \be_u$. 
However, the constant $C$ in the estimate \eqref{unistab2} depends on the constants $\Omega$ and $\eps$ that participate in the constraints \eqref{boundabove} and \eqref{RBbound} involving complete vectors $v_u$ and $\tilde v_u$ (i.e., also depending on non-diagonal elements of $\be_u$ and $\tilde \be_u$). Clearly, $Z_j < \infty$ ($j = \overline{1,m}$) for the partition $J_{2n-1} := \{ (n,1) \}$, $J_{2n} := \{ (n,k) \}_{k = 2}^m$ by virtue of the asymptotics \eqref{asymptla2} and \eqref{asymptbe2}. Analogously to Theorem~\ref{thm:unistab}, one can formulate the uniform stability of the recovery of arbitrary matrix potential $Q \in \mathcal P$ (not necessarily diagonal) from the spectral data.

Theorems~\ref{thm:boundinv2} and~\ref{thm:unistab2} are proved by using a linear main equation analogous to \eqref{main}. Let us outline the construction of the main equation for the boundary value problem $\mathscr L = \mathscr L(Q)$ of form \eqref{eqv2}--\eqref{bc2} with $Q \in \mathcal P$ by repeating the arguments of Section~\ref{sec:maineq}.

Introduce the model problem $\tilde{\mathscr L} := \mathscr L(0_m)$. It has the spectral data
\begin{gather*}
    \tilde \la_{nk} = \tilde \rho_{nk}^2, \quad \tilde \rho_{n1} = n-\tfrac{1}{2}, \quad \tilde \rho_{nk} = n, \: k = \overline{2,m}, \quad n \ge 1, \\
    \tilde \al_{n1} = \tilde \be_{n1} = \tfrac{2}{\pi} \left( n - \tfrac{1}{2}\right)^2 T, \quad \tilde \al_{nk} = \tilde \be_n = \tfrac{2}{\pi} n^2 T^{\perp}, \: k = \overline{2,m}, \quad n \ge 1.
\end{gather*}
For brevity, denote
$$
\tilde \rho_n := \tilde \rho_{nk}, \quad \tilde \la_n = \tilde \la_{nk}, \: k \ge 2, \quad \hat \rho_{nk} := \rho_{nk} - \tilde \rho_{nk}, \: k = \overline{1,m}, \quad n \ge 1.
$$

According to the initial conditions \eqref{icsin}, we have
\begin{equation} \label{vvtsin}
\tilde \vv(x, \rho^2) = \frac{\sin \rho x}{\rho} I_m.
\end{equation}
In addition, define the function $\tilde D(x, \mu, \la)$ by \eqref{defDt} using \eqref{vvtsin}. Then, the following relation holds (see \cite[Lemma~3.2]{Bond20-ipse}):
\begin{multline} \label{sumrelvv2}
\tilde \vv(x, \la) = \vv(x, \la) + \sum_{l = 1}^{\infty} \Biggl( \sum_{s = 1}^m \vv(x, \la_{ls}) \be_{ls} \tilde D(x, \la_{ls}, \la) \\ - \vv(x, \tilde\la_{l1}) \tilde \be_{l1} \tilde D(x, \tilde \la_{l1}, \la) - \vv(x, \tilde \la_l) \tilde \be_l \tilde D(x, \tilde \la_l, \la) \Biggr),
\end{multline}
where the series converges with brackets absolutely and uniformly by $x \in [0,\pi]$ and $\la$ on compact sets.

Introduce the matrix functions $w_{nk}(x, \rho)$, $\tilde w_{nk}(x, \rho)$, and $\tilde W_{nk}(x, \theta, \rho)$ analogously to $w_n(x, \rho)$, $\tilde w_n(x, \rho)$, and $\tilde W_n(x, \theta, \rho)$ given by \eqref{defwn}--\eqref{defWnt}, respectively, by replacing $\tilde \rho_n$ with $\tilde \rho_{nk}$ and using $\vv(x, \la)$ satisfying \eqref{icsin}. For $n \ge 1$, denote
\begin{gather*}
\psi_{nk}(x) := n w_{nk}(x, \rho_{nk}), \quad k = \overline{1,m}, \\
\psi_{n,m+1}(x) := n \vv(x, \tilde \la_{n1}), \quad \psi_{n,m+2}(x) := n \vv(x, \tilde \la_n).
\end{gather*}
Analogously, define $\tilde \psi_{nk}$ by using $\tilde w_{nk}$ and $\tilde \vv$. From \eqref{sumrelvv2}, we derive the system
\begin{equation} \label{sys2}
\tilde \psi_{nk}(x) = \psi_{nk}(x) + \sum_{l = 1}^{\infty} \sum_{s = 1}^{m+2} \psi_{ls}(x) \tilde{\mathcal R}_{ls,nk}(x), \quad n \ge 1, \, k = \overline{1,m+2},
\end{equation}
where
\begin{align*}
\tilde{\mathcal R}_{ls,nk}(x) & := n l^{-1}\hat \rho_{ls} \be_{ls} \tilde W_{nk}(x, \rho_{ls}, \rho_{nk}), \quad s,k = \overline{1,m}, \\
\tilde{\mathcal R}_{ls,n,m+1}(x) & := n l^{-1} \hat \rho_{ls} \tilde D(x, \la_{ls}, \tilde \la_{n1}), \quad s = \overline{1,m}, \\
\tilde{\mathcal R}_{ls,n,m+2}(x) & := n l^{-1} \hat \rho_{ls} \be_{ls} \tilde D(x, \la_{ls}, \tilde \la_n), \quad s = \overline{1,m}, \\
\tilde{\mathcal R}_{l,m+1,nk}(x) & := n l^{-1} \bigl( \be_{l1} \tilde W_{nk}(x, \rho_{l1}, \rho_{nk}) - \tilde \be_{l1} \tilde W_{nk}(x, \tilde \rho_{l1}, \rho_{nk})\bigr), \quad k = \overline{1,m}, \\
\tilde{\mathcal R}_{l,m+1,n,m+1}(x) & := n l^{-1} \bigl(\be_{l1} \tilde D(x, \la_{l1}, \tilde \la_{n1}) - \tilde \be_{l1} \tilde D(x, \tilde \la_{l1}, \tilde \la_{n1})\bigr), \\
\tilde{\mathcal R}_{l,m+1,n,m+2}(x) & := n l^{-1} \bigl(\be_{l1} \tilde D(x, \la_{l1}, \tilde \la_{n}) - \tilde \be_{l1} \tilde D(x, \tilde \la_{l1}, \tilde \la_{n})\bigr), \\
\tilde{\mathcal R}_{l,m+2,nk}(x) & := n l^{-1} \biggl( \sum_{s = 2}^m \be_{ls} \tilde W_{nk}(x, \rho_{ls}, \rho_{nk}) - \tilde \be_l \tilde W_{nk}(x, \tilde \rho_{ls}, \rho_{nk}) \biggr), \\
\tilde{\mathcal R}_{l,m+2,n,m+1}(x) & := n l^{-1} \biggl( \sum_{s = 2}^m \be_{ls} \tilde D(x, \la_{ls}, \tilde \la_{n1}) - \tilde \be_l \tilde D(x, \tilde \la_{ls}, \tilde \la_{n1}) \biggr), \\
\tilde{\mathcal R}_{l,m+2,n,m+2}(x) & := n l^{-1} \biggl( \sum_{s = 2}^m \be_{ls} \tilde D(x, \la_{ls}, \tilde \la_{n}) - \tilde \be_l \tilde D(x, \tilde \la_{ls}, \tilde \la_{n}) \biggr).
\end{align*}

The multipliers $n$ and $n l^{-1}$ in the definitions of $\psi_{nk}$, $\tilde \psi_{nk}$ and $\tilde{\mathcal R}_{ls,nk}$, respectively, are added for achieving the estimates \eqref{estpsi} and \eqref{estpsiRt} for these matrix functions with 
$$
\xi_n := |\rho_{n1} - \tilde \rho_{n1}| + \sum_{k = 2}^m |\rho_{nk} - \tilde \rho_n| + \| \be_{n1} - \tilde \be_{n1} \| + \| \be_n - \tilde \be_n \|, \quad \{ n \xi_n \}_{n \ge 1} \in l_2.
$$

Thus, the system \eqref{sys2} can be interpreted as a linear equation in the Banach space of bounded infinite matrix sequences $a = \{ a_{nk} \}_{n \ge 1, \, k = \overline{1,m+2}}$, $a_{nk} \in \mathbb C^{m \times m}$, with the norm $\| a \| = \sup_{n,k} \| a_{nk} \|$. We obtain the main equation of the same form as \eqref{main} and the result analogous to Theorem~\ref{thm:maineq} on its solvability, which allows us to prove Theorem~\ref{thm:boundinv2}.

For proving the uniform stability, consider two problems $\mathscr L$ and $\tilde{\mathscr L}$ with potentials $Q, \tilde Q \in \mathcal P$, where $\tilde Q$ is not necessarily zero. By virtue of \cite[Theorem~3.6]{Bond20-ipse}, the relation analogous to \eqref{relQhH}, \eqref{defE0} holds for the potentials:
\begin{equation} \label{difQ}
Q(x) - \tilde Q(x) = -2 \frac{d}{dx} \left( \sum_{n = 1}^{\infty} \sum_{k = 1}^m \bigl(\vv(x, \la_{nk}) \be_{nk} \tilde \vv^*(x, \la_{nk}) - \vv(x, \tilde \la_{nk}) \tilde \be_{nk} \tilde \vv^*(x, \tilde \la_{nk})\bigr)\right),
\end{equation}
where $\vv(x, \la)$ is specified by the initial conditions \eqref{icsin} and the numbers $\be_{nk}$ are defined by \eqref{defvnk}.

If the matrix potential $Q$ is diagonal, then so is the solution $\vv(x, \la)$. For $j = \overline{1,m}$, denote by $\vv_j(x, \la)$ and $\be_{nkj}$ the diagonal elements at the position $(j,j)$ of the matrices $\vv(x, \la)$ and $\be_{nk}$, respectively. If $Q, \tilde Q \in \mathcal P^d$, then the relation \eqref{difQ} implies
\begin{equation} \label{difqj}
q_j(x) - \tilde q_j(x) = -2 \frac{d}{dx} \left( \sum_{n = 1}^{\infty} \sum_{k = 1}^m \bigl(\be_{nkj} \vv_j(x, \la_{nk}) \tilde \vv_j(x, \la_{nk}) - \tilde \be_{nkj} \vv_j(x, \tilde \la_{nk}) \tilde \vv_j(x, \tilde \la_{nk})\bigr)\right)
\end{equation}
for $j = \overline{1,m}$. The relation \eqref{difqj} can be used to obtain the estimate \eqref{unistab2} similarly to Lemma~\ref{lem:estE0}.

\medskip

{\bf Acknowledgments.} This work was supported by Grant 24-71-10003 of the Russian Science Foundation, https://rscf.ru/en/project/24-71-10003/.

\noindent Natalia Pavlovna Bondarenko \\

\noindent 1. Department of Mechanics and Mathematics, Saratov State University, 
Astrakhanskaya 83, Saratov 410012, Russia, \\

\noindent 2. S.M. Nikolskii Mathematical Institute, RUDN University, 6 Miklukho-Maklaya St, Moscow, 117198, Russia, \\

\noindent 3. Moscow Center of Fundamental and Applied Mathematics, Lomonosov Moscow State University, Moscow 119991, Russia.\\

\noindent e-mail: {\it bondarenkonp@sgu.ru}

\end{document}